\documentclass[titlepage,12pt]{amsart}
\usepackage{amsaddr}
\usepackage{latexsym}
\usepackage{a4}
\usepackage{amssymb}
\usepackage{setspace}
\usepackage{float}
\usepackage[english]{babel}
\usepackage{url}

\newtheorem{thm}{Theorem}[section] 
 
\newtheorem{lem}[thm]{Lemma}

\theoremstyle{definition}
 
\newtheorem{de}[thm]{Definition}
\numberwithin{equation}{section}

\newcommand{\N}{\mathbb{ N}} 
\newcommand{\Z}{\mathbb{ Z}} 
\newcommand{\R}{\mathbb{ R}}

\newcommand{\ol}[1]{\overline{#1}}

\newcommand{\algop}[2]{( {#1}, {#2} )}

  \newcommand{\noSp}[2]{\genfrac{}{}{0pt}{}{#1}{#2}}
  \newcommand{\SUM}[3]{{\sum\limits_{#1}^{#2} #3}} 
  \newcommand{\eqclass}[2]{{[ \equiv_{#1}; \; #2 ]}}
  \newcommand{\lcm}{\mathrm{lcm}}
  \newcommand{\coef}{\mathrm{coeff}}

\title{Subnearrings of $\algop{\Z[x]}{+,\circ}$} 
\author{Erhard Aichinger}
\address{Erhard Aichinger,
Institut f\"ur Algebra,
Johannes Kepler Universit\"at Linz,
Altenbergerstra{\ss}e 69,
4040 Linz,
Austria}
\email{\tt erhard@algebra.uni-linz.ac.at}

\author{Sebastian Kreinecker}
\address{Sebastian Kreinecker,
Institut f\"ur Algebra,
Johannes Kepler Universit\"at Linz,
Altenbergerstra{\ss}e 69,
4040 Linz,
Austria}
\email{\tt kreinecker@algebra.uni-linz.ac.at}

\thanks{Supported by the Austrian Science Fund (FWF):P29931}
\date{\today}

\begin{document}
\bibliographystyle{amsplain}
\begin{abstract}
  We show that the nearring $\algop{\Z[x]}{+,\circ}$ of integer polynomials,
  where the nearring multiplication is the composition of polynomials,
  has uncountably
  many subnearrings, and we give an explicit description of those
  nearrings that are generated by subsets of $\{1,x,x^2,x^3\}$.
\end{abstract} 

\maketitle

\section{Motivation and results}

A basic algebraic operation on polynomials is
\emph{composition}: the composition of the polynomials $f$ and $g$
is defined by
$f \circ g \, (x) := f (g(x))$. For example,
$(x^2 + 2) \circ (2 x^3 - 1)=  (2 x^3-1)^2+2 = 4 x^6-4 x^3+3$.
This operation $\circ$ is associative and satisfies the right distributive
law $(f + g) \circ h = f \circ h + g \circ h$, and therefore
$\algop{\Z[x]}{+, \circ}$ is a nearring \cite{Pi:Near}.
The ideal structure of this nearring was investigated in
\cite{Br:CAOP, Gu:AADL, GR:IITN}. In the present paper,
we provide some information on the set of subalgebras of this
nearring.
For this purpose, we describe the structure of
the subnearrings generated by some simple sets of polynomials,
such as the singleton $\{x^2\}$. In this nearring, we find
the polynomial $x^2 \circ (x^2 + x^2 \circ x^2 \circ x^2) -
x^2 \circ x^2 -  x^2 \circ x^2 \circ x^2 \circ x^2 =
(x^2 + x^8)^2 - x^4 - x^{16} = 2x^{10}$, and we will see
that $x^{10}$ is not an element of this nearring.
We will compute the nearring generated by each of the $16$ subsets
of $\{1,x,x^2, x^3\}$ in Section~\ref{sec:Examples}.
In other words, we describe those polynomials that can be
obtained from, say, $x^2$ by using only addition, subtraction and
composition.
Taking a broader view, we will see
that there are continuum many subnearrings of $\algop{\Z[x]}{+,\circ}$
by detecting an infinite independent subset inside the nearring $\Z[x]$ (Theorem \ref{thm:indep}).

\section{Preliminaries on nearring generation}

In this section, we state three lemmas that facilitate
the description of the nearring generated by a given set of
elements. For a nearring $\algop{N}{+,\circ}$ and a subset
$F$ of $N$, we let $\langle F \rangle$
denote the subnearring generated by $F$.
We are mainly concerned with $\algop{N}{+, \circ} =
\algop{\Z[x]}{+,\circ}$. If $F = \{f(x)\}$ for
some $f(x) \in \Z[x]$, then it is easy to see that $\langle F \rangle$
is contained in the subnearring
$\Z[f(x)] := \{q (f(x)) \mid q \in \Z[x]\}$ of $\algop{\Z[x]}{+,\circ}$.
However, we will see that $\Z[f(x)]$ is often strictly larger than
$\langle f(x) \rangle$. 
The following lemma helps to establish that a given set
$M$ is indeed the nearring generated by $F$.
\begin{lem}[{\cite[Lemma~2.1]{Ai:GPUF}}] \label{lem:genLem1} 
Let $\algop{N}{+, \circ}$ be a nearring, and let $F$ and $M$ be non-empty subsets of $N$. 
We assume that the following conditions hold:
\begin{enumerate}
\item $F \subseteq M$.
\item $M \subseteq \langle F \rangle$.
\item $M$ is closed under $+$ and $-$.
\item $F \circ M \subseteq M$.
\end{enumerate} 
Then $M = \langle F \rangle$.
\end{lem}

Following \cite[Definitions~1.7~and~1.11]{Pi:Near}, we call
an element $a$
in a nearring $\algop{N}{+,\circ}$ 
\emph{right cancellable in $N$} if for all $x, y \in N$ with $x \circ a =
y \circ a$, we have $x = y$; using the distributive law,
one obtains 
that $a$ is right cancellable if and only if $x \circ a = 0$ implies $x = 0$
for all $x \in N$. An element $e \in \N$ is a \emph{left identity} if
$e \circ x = x$ for all $x \in N$, and an element $c \in N$ is \emph{constant}
if $c \circ 0 = c$. If $c$ is constant, then $c \circ x = c$
for all $x \in N$. The set of constant elements of $N$ will be denoted
by $C(N)$.
In the nearring $\algop{\Z[x]}{+,\circ}$, every $p$ with
$\deg (p) \ge 1$ is right cancellable, the polynomial $x$ is
the only left identity, and the constant elements of
$\Z[x]$ are given by $C(\Z[x]) = \{ c_0  x^0 \mid c_0 \in \Z \}$.
If $e$ is a left identity and $F$ a subset of $N$, then sometimes
the nearrings $\langle F \rangle$ and
$\langle F \cup \{e\} \rangle$ are closely related:
\begin{lem} \label{lem:genLem2}
  Let $N$ be a nearring, let $a \in N$ be a right cancellable element
  of $N$, let $e$ be a left identity of $N$, and let $C$ be a
  subset of $C(N)$.
  Let $F$ be the subnearring of $N$ generated by $\{a\} \cup C$, and let
  $G$ be the subnearring of $N$ generated by $\{a, e\} \cup C$.
  Then $G = \{ n \in N \mid n \circ a \in F\}$.
\end{lem}
\begin{proof}
  For this proof, we resort to some notions from universal algebra
  \cite{BS:ACIU}. In the setting of universal algebra,
  we see $N$ as an algebra $\algop{N}{+^N,-^N,0^N, \circ^N}$,
  where $+,-,0,\circ$ are the operation symbols and
  $+^N, -^N, 0^N, \circ^N$ denote their interpretations as
  finitary operations on $N$.
  The subnearring generated by a subset $B$ of $N$
  can then be described as the set
  of all $T^N (b_1,\ldots, b_k)$, where $k \in \N_0$,  $T$ is a term over the variables
  $X_1, \ldots,X_k$ in the language $\{+,-,0,\circ\}$, and $b_1, \ldots, b_k \in B$ \cite[Theorem~II.10.3(c)]{BS:ACIU}.
  Here $T^N$ denotes the term function that the term $T$ induces on $N$.
  For $k \in \N_0$, let  $F(X,Y_1,\ldots, Y_k) = F(X, \ol{Y})$ denote the set of terms over the variables $X, Y_1, \ldots, Y_k$,
  and $F(X,Y_1,\dots,Y_k,Z) = F(X, \ol{Y}, Z)$ denote the set of terms over the variables $X, Y_1, \dots, Y_k, Z$
  in the language of nearrings.
  Our first claim is:
  \begin{multline} \label{eq:c1}
    \forall k \in \N_0 \,
    \forall T \in F(X,\ol{Y},Z) \, \forall S \in F(X, \ol{Y}) \, \exists U \in F(X, \ol{Y}) \,
     \forall \ol{c} \in C(N)^k \,:\, \\
    T^{N} (a, \ol{c}, e) \, \, \circ^N \, S^N (a, \ol{c}) = U^N(a, \ol{c}).
  \end{multline}
  We fix $k \in \N_0$ and 
  proceed by induction on the depth of the term $T$.
  If $T = 0$, then set $U := 0$.
  If $T = X$, then we let $U(X, \ol{Y}) := X \circ S(X, \ol{Y})$.
  We fix $\ol{c} \in C(N)^k$ 
  and compute
  $T^{N} (a, \ol{c}, e) \, \circ^N S^N (a, \ol{c}) =
  a \, \circ^N S^N (a, \ol{c}) = U^N (a, \ol{c})$.
  If $T = Y_j$ with $j \in \{1,\ldots, k\}$, we set $U(X, \ol{Y}) := Y_j \circ S(X, \ol{Y})$ (note that
  $U(X,\ol{Y}) := Y_j$ would also work), and if 
  $T = Z$, we set $U := S$.
  
  For the induction step, suppose first that $T = T_1 + T_2$.
  The induction hypothesis yields $U_1, U_2 \in F(X, \ol{Y})$ with
  $T^N_i (a,\ol{c},e) \, \circ^N S^N (a,\ol{c}) = U^N_i (a,\ol{c})$ for $i \in \{1,2\}$ and for
  all $\ol{c} \in C(N)^k$, and thus
  $T^N (a,\ol{c},e) \, \circ^N S^N (a,\ol{c})  =
  (T_1^N (a,\ol{c},e) +^N T_2^N (a,\ol{c},e)) \, \circ^N S^N (a,\ol{c}) =
   T_1^N (a,\ol{c},e) \, \circ^N S^N (a,\ol{c}) \, +^N \, T_2^N(a,\ol{c},e) \circ S^N (a,\ol{c}) =
   U_1^N (a,\ol{c}) +^N U_2^N (a,\ol{c})$ for all $\ol{c} \in C(N)^k$.
  Now $U := U_1 + U_2$ is the
  required term.
  The case $T = T_1 - T_2$ is done similarly.
  Next, we suppose that $T = T_1 \circ T_2$.
  Then for all $\ol{c} \in C(N)^k$, we have $T^N (a,\ol{c},e) \, \circ^N S^N (a,\ol{c}) =
  (T_1^N (a,\ol{c},e) \, \circ^N T_2^N (a,\ol{c},e)) \, \circ^N S^N (a,\ol{c}) =
  T_1^N (a,\ol{c},e) \, \circ^N (T_2^N (a,\ol{c},e) \, \circ^N S^N (a,\ol{c}))$. 
  By the induction hypothesis, we have $U_2 \in F(X,\ol{Y})$ with
  $T_2^N (a,\ol{c},e) \, \circ^N S^N (a,\ol{c}) = U^N_2 (a,\ol{c})$ for all $\ol{c} \in C(N)^k$. Thus
  $T_1^N (a,\ol{c},e) \, \circ^N (T_2^N (a,\ol{c},e) \, \circ^N S^N (a,\ol{c})) = T_1^N (a,\ol{c},e)
  \, \circ^N U_2^N (a,\ol{c})$ for all $\ol{c} \in C(N)^k$.
  Using the induction hypothesis again
  (for $T_1$), we obtain $U \in F(X,\ol{Y})$ with
  $T_1^N (a,\ol{c},e)
  \, \circ^N U_2^N (a,\ol{c}) = U^N (a,\ol{c})$.
  This completes the proof of claim~\eqref{eq:c1}.
  Now we are ready to prove
  \begin{equation} \label{eq:i1}
    G \subseteq \{ n \in N \mid n \, \circ^N a \in F\}.
  \end{equation}
  To this end, let $g \in G$. Then there is $k \in \N_0$, there are
  $c_1, \ldots, c_k \in C$, and there is $T \in F(X,\ol{Y},Z)$ with
  $g = T^N (a,\ol{c},e)$. Our goal is to show that $g \, \circ^N a  \in F$.
  To this end, we observe that $g \, \circ^N a = T^N (a,\ol{c},e) \, \circ^N S^N  (a,\ol{c})$
  where $S(X,\ol{Y}) = X$. Now claim~\eqref{eq:c1} yields $U \in F(X,\ol{Y})$ with
  $T^N (a,\ol{c},e) \, \circ^N S^N  (a,\ol{c}) = U^N (a,\ol{c})$. Since $U^N (a,\ol{c}) \in F$,
  we obtain $g \, \circ^N a \in F$, which completes the proof of~\eqref{eq:i1}.

  For establishing the converse inclusion, we define for each $k \in \N_0$
  a mapping
  $R : F(X, \ol{Y}) \to F(X,\ol{Y},Z)$ inductively by $R(0) := 0$,
  $R(X) := Z$, $R(Y_j) := Y_j$ for $j \in \{1,\ldots, k\}$,
  $R(T_1 \circ T_2) := T_1 \circ R(T_2)$,
  $R(T_1 + T_2) := R(T_1) + R(T_2)$, $R (T_1 - T_2) := R(T_1) - R(T_2)$.
  Now we claim:
  \begin{equation} \label{eq:c2}
    \forall T \in F(X, \ol{Y}) \, \forall \ol{c} \in C(N)^k \,:\, R(T)^N (a,\ol{c},e) \, \circ^N a = T^N (a,\ol{c}).
  \end{equation}
  We proceed by induction on the depth of $T$.
  For $T = 0$, both sides of~\eqref{eq:c2} evaluate to $0^N$.
  For $T = X$, we compute
  $R(X)^N (a,\ol{c},e) \, \circ^N a = e \, \circ^N a =
  a = T^N (a,\ol{c})$.
  For $T = Y_j$,  we notice that $R(T) = Y_j$ and compute
  $R(T)^N (a,\ol{c},e) \, \circ^N a = c_j \, \circ^N a$. Using
  that $c_j$ is constant, we have $c_j \, \circ^N a = c_j = T^N (a, \ol{c})$.
  
  For the induction step, we see that the cases
  $T = T_1 + T_2$ and $T = T_1 - T_2$ can be done with
  routine calculations using
  the induction hypothesis and the right distributive law.
  Now let $T = T_1 \circ T_2$. Then we compute
  $(R (T_1 \circ T_2))^N (a, \ol{c}, e) \, \circ^N a =
   (T_1 \circ R(T_2))^N (a, \ol{c}, e) \, \circ^N a =
   (T_1^N (a, \ol{c}) \, \circ^N R(T_2)^N (a, \ol{c}, e)) \, \circ^N a =
    T_1^N (a, \ol{c}) \, \circ^N (R(T_2)^N (a, \ol{c}, e) \, \circ^N a)$.
  Using the induction hypothesis, the last expression is equal to
   $T_1^N (a,\ol{c}) \, \circ^N T_2^N (a,\ol{c}) = (T_1 \circ T_2)^N (a,\ol{c}) = T^N (a,\ol{c})$.
  This completes the proof of claim~\eqref{eq:c2}.
  Next, we prove
  \begin{equation} \label{eq:i2}
    \{ n \in N \mid n \, \circ^N a \in F\} \subseteq G.
  \end{equation}
  Let $n \in N$ be such that $n \, \circ^N a \in F$. Then
  there is $k \in \N_0$, there are $c_1, \ldots, c_k \in C$ and there is
  a term $T \in F(X,\ol{Y})$ such that $n \, \circ^N a = T^N (a,\ol{c})$.
  Hence, by claim~\eqref{eq:c2},
  $n \, \circ^N a = R(T)^N (a,\ol{c},e) \, \circ^N a$. Since $a$ is right
  cancellable, $n = R(T)^N (a,\ol{c}, e)$, and therefore $n$ lies
  in the nearring generated by $\{a, c_1, \ldots , c_k, e\}$, and thus in $G$.
  \end{proof}
Setting $C := \emptyset$, we obtain the following consequence:
 \begin{lem} \label{lem:genLem3}
  Let $N$ be a nearring, let $a \in N$ be a right cancellable element
  of $N$, let $e$ be a left identity of $N$.
  Let $F$ be the subnearring of $N$ generated by $\{a\}$, and let
  $G$ be the subnearring of $N$ generated by $\{a, e\}$.
  Then $G = \{ n \in N \mid n \circ a \in F\}$.
\end{lem}

\section{Preliminaries from elementary number theory} 

The proof of Theorem \ref{thm:indep} will rely
on number theoretic properties of the multinomial coefficients.
For $p \in \Z[x]$ and $i \in \N_0$, we denote  
the coefficient of $x^i$ in the polynomial $p$ by $\coef(p, i)$,
where $\coef(p, i) =0$ if $i> \deg(p)$.
For $k = k_1 + \dots + k_n$, we 
abbreviate  the \emph{multinomial coefficient}
$\frac{k!}{k_1! \cdots k_n!}$ by
$\binom{k}{k_1, \dots , k_n}$.
 
\begin{lem}[Corollary 32.1 of \cite{Si:DOBAMC}] \label{lem:pdividesmultinom}
 Let $p \in \mathbb{P}$, let $n, s \in \N$, let $j \in \{1,\ldots, n\}$, and let $k,  k_1, \ldots, k_n \in \N_0$ be such that $k= k_1 + \ldots + k_n$.
 If $p^s$ divides $k$ and $\gcd(k_j, p) = 1$, then $p^s$ 
 divides $\binom{k}{k_1, \ldots, k_n}.$
\end{lem}

\begin{lem} \label{lem:2divmult}
  Let $n, m \in \N$, and
  let $l_1, \ldots, l_n, k_1, \ldots, k_n, k \in \N_0$ 
  be
  such that $k= k_1 + \ldots + k_n$, $k \equiv_2 0$, 
  and for all $i \in \{1,\ldots, n\}$, $l_i \equiv_2 0$.
  We assume that
 $2^{m + 1}-2= l_1 k_1 + \ldots +l_n k_n$.
 Then $2$ divides $\binom{k}{k_1, \ldots, k_n}.$
\end{lem}

\begin{proof}
  Seeking a contradiction, we suppose that
  $\binom{k}{k_1, \ldots, k_n}$ is odd.
  Then we get from Lemma \ref{lem:pdividesmultinom} that either
  $k$ is odd, or for all $j \leq n$, $k_j$ is even. 
  Since $k$ is even by assumption, all $k_j$ are even.
   Since $l_1, \ldots, l_n$ are also even, $4$ divides $l_1 k_1 + \ldots +l_n k_n$.
 But $4$ does not divide  $2^{m + 1} -2$, which contradicts the assumption
 $2^{m + 1}-2 = l_1 k_1 + \ldots +l_n k_n.$
\end{proof}

\section{Uncountably many subnearrings} \label{sec:unManySub}
We let
$P (\N)$ denote the power set of $\N$.
\begin{thm} \label{thm:indep}
  For $i \in \N$, let $p_i := x^{2^{i+1} - 2}$, let $S$ be the set of
  all subnearrings of $\algop{\Z[x]}{+, \circ}$ and let
  $\Phi : P(\N) \to S$, where for each $A \subseteq \N$, $\Phi (A)$ is the subnearring
  of $\Z[x]$ that is generated by $\{ p_i \mid i \in A \}$. 
  Then we have:
  \begin{enumerate}
  \item \label{it:t1} For each $j \in \N$, $p_j$ is not an element of the nearring
        $\Phi (\N \setminus \{j\})$.
  \item \label{it:t2} For all $A,B \in P(\N)$ we
    have $\Phi (A) \subseteq \Phi (B)$ if and only if $A \subseteq B$. In particular,
      $\Phi$ is injective and hence $|S| = 2^{\aleph_0}$.
  \item \label{it:t3} $(S , \subseteq)$ contains a subset that is order isomorphic
    to $(P(\N), \subseteq)$.
  \end{enumerate}  
\end{thm}

\begin{proof} %
  \eqref{it:t1}
Let $j \in \N$, let $\mathbb{K} =  \N \setminus \{j\} $ and
let $N^* := \langle \{p_i \mid i \in \mathbb{K} \}  \rangle$.
We show that for each $p \in N^*$,
$\coef(p, 2^{j+1}-2)$ is even. To this end, let
$p \in N^*$.
Next, we define infinitely many unary operations on $\Z[x]$;
actually, for each $i \in \mathbb{K}$, we define a unary
operation $\overline{p_i}$ on $\Z[x]$ by
$\overline{p_i}(q) = p_i \circ q$
for $q \in \Z[x]$.
Let $M$ be the subalgebra of
\[
     \algop{\Z[x]}{+, -, 0, (\overline{p_i})_{i \in \mathbb{K}}}
\]
that is generated by $F = \{p_i \mid i \in \mathbb{K}\}$.
Then  Lemma \ref{lem:genLem1} yields $N^* = M$.
Hence there exists a term $T_p$ 
 over the variables $(X_i)_{i \in \mathbb{K}}$ in the language of $M$, which has
 $+, -, ({\overline{p_i}})_{i \in \mathbb{K}} $ as operation symbols, such that $p={T_p}^M ( (p_i)_{i \in \mathbb{K}} )$.

 We will now use induction on the depth of $T$ to show that for
 each term $T$, the coefficient of $x^{2^{j+1}-2}$ in the polynomial
 $T^M ( (p_i)_{i \in \mathbb{K}} )$ is even.
 If $T = X_i$ with $i \in \mathbb{K}$, then $T^M((p_i)_{i \in \mathbb{K}}) = p_i$,
 and the coefficient of $x^{2^{j+1}-2}$ in $p_i$ is $0$, hence even.
 If $T = 0$, the assertion is obvious.
For the induction step, suppose first that $T= T_1 + T_2$ or $T= T_1 - T_2$. 
Then the induction hypothesis yields $\coef(T_1^M((p_i)_{i \in \mathbb{K}}), 2^{j+1}-2) \equiv_2 0$ and 
$\coef(T_2^M((p_i)_{i \in \mathbb{K}}), 2^{j+1}-2) \equiv_2 0$.
Therefore we have $\coef(T^M((p_i)_{i \in \mathbb{K}}), 2^{j+1}-2) \equiv_2 0$.
Finally,  we suppose that there exists $i \in \mathbb{K}$ such that $T = \overline{p_i}(T_1)$.
This means $p = x^{2^{i+1} - 2} \circ T_1^M((p_i)_{i \in \mathbb{K}})$.
 Since $M \subseteq \Z[x^2]$ and $T_1^M((p_i)_{i \in \mathbb{K}}) \in M$,
 all monomials of $T_1^M((p_i)_{i \in \mathbb{K}})$ have even exponents, and
 therefore, there are $n \in \N$ and $c_0, \ldots, c_n \in \Z$ 
such that $T_1^M((p_i)_{i \in \mathbb{K}}) = \sum_{i=0}^{n} c_i x^{2i}$.
By the multinomial theorem we get
\[
    \left( \SUM{j=0}{n}{c_j x^j} \right)^{2^{i+1} - 2} = 
    \SUM{{ \begin{array}{c} (k_0, \ldots, k_n) \in \N_0^{n+1} \\k_0 + \cdots + k_n = 2^{i+1} - 2
         \end{array}}}{}{ \binom{2^{i+1} - 2 }{k_0, \ldots, k_n} \prod\limits_{r=0 }^{n} c_r^{k_r}x^{2 r k_r}   }.
 \]
    A summand of the right hand side that contributes to the coefficient
    of $2^{j+1}-2$ comes from a $(k_0, \ldots, k_n)$ such that
    $\sum_{r=0}^n 2 r  k_r = 2^{j+1} - 2$. By Lemma~\ref{lem:2divmult} (with
    $\tilde{k} := 2^{i+1} - 2$),  for such a $(k_0, \ldots, k_n)$,
      the multinomial coefficient $\binom{2^{i+1} - 2 }{k_0, \ldots, k_n}$
      is even.
 Thus
\[
  \coef((\overline{p_i} (T_1))^M((p_i)_{i \in \mathbb{K}}), 2^{j+1}-2) \equiv_2 0.
\]
This concludes the induction step.
Hence the coefficient of $x^{2^{j+1}-2}$ in
$p  = {T_p}^M((p_i)_{i \in \mathbb{K}})$ is even.
Therefore,
$p_j = 1 x^{2^{j+1}-2} \not\in N^*$.

\eqref{it:t2} The ``if''-direction is obvious. For the ``only if''-direction,
we assume $A \not\subseteq B$, and we let $j \in A$ such that $j \not\in B$.
Then $p_j \in \Phi (A)$ and by item~\eqref{it:t1},
$p_j \not\in \Phi (\N \setminus \{j\})$. Since
$\Phi (B) \subseteq \Phi (\N \setminus \{j\})$, we obtain $p_j \not\in \Phi(B)$,
and therefore $\Phi(A) \not\subseteq \Phi(B)$.

\eqref{it:t3} By~\eqref{it:t2}, the image of item~\eqref{it:t3}  $\Phi$ is order isomorphic
to $(P(\N), \subseteq)$.

\end{proof}
    As a consequence, the set of subnearrings of $\algop{\Z[x]}{+,\circ}$
    contains
    subsets of each of the following order types:
   \emph{infinite ascending chains}, i.e., subsets order isomorphic
      to $\algop{\N}{\le}$,
   \emph{infinite descending chains}, i.e., subsets order isomorphic
      to $\algop{\{z \in \Z \mid z < 0\}}{\le}$,
    \emph{uncountable linearly ordered subsets that are order isomorphic
      to $\algop{\R}{\le}$}, and 
    \emph{uncountable antichains}, i.e., subsets order isomorphic
      to $\algop{\R}{=}$.
      As another consequence, the nearring $\Phi (\N)$, which is the nearring
      generated by $\{p_i \mid i \in \N\}$, is not finitely
generated.
However, the entire nearring $\Z[x]$  is finitely generated,
with generators $\{1, x, x^2, x^3\}$ \cite[Corollary~4.3]{Ai:GPUF}.
An example of a descending chain of finitely generated subnearrings
of $\algop{\Z[x]}{+, \circ}$ is provided by $(N_i)_{i \in \N}$,
where $N_i:= \langle \{ x^{2^{2^i}} \} \rangle$.
Then for each $i \in \N$,
$x^{2^{2^{i+1}}}
x^{2^{2^i+2^i}} =
x^{2^{2^i} \cdot 2^{2^i}} =
x^{2^{2^i}} \circ x^{2^{2^i}} \in N_i$, and therefore $N_{i+1} \subseteq N_i$.
The inclusion is proper because
$N_{i+1}$ does not contain any nonzero polynomial $p \in \Z[x]$ with $\deg(p) < 2^{2^{i+1}}$,
and thus $x^{2^{2^i}} \not \in N_{i+1}$.

\section{Examples of subnearrings of $\algop{\Z[x]}{+, \circ}$} \label{sec:Examples}

In this section we describe all subnearrings of $\algop{\Z[x]}{+, \circ}$
which are generated by an arbitrary subset of $\{1, x, x^2, x^3\}$.
Let $ i, j \in \N$ and $y_1, \ldots , y_j \in \N_0$. 
We denote the set ${\{ x \in \N_0 \; \big\vert \; \exists k \in \{1, \ldots, j \}: \; x \equiv_{i} y_k   \} }$
by $\eqclass{i}{y_1, \ldots , y_j}$ and the digit sum of the number $a$ in base $b$ by $s_b(a)$.

For the remainder of this section let $A, B, C, D$ defined by
\begin{equation*}
\begin{aligned}
 &A:= \eqclass{24}{15, 21}, \\
 &B:= { \eqclass{72}{3, 33 , 45, 51, 57, 63} \setminus \{3\} }, \\
 &C:= \eqclass{8}{5, 7} \text{ and } \\
 &D:= { \eqclass{24}{1, 11 , 15, 17, 19, 21} \setminus \{1\} }.
 \end{aligned}
\end{equation*}

The following table shows all possible combinations of subnearrings of $\algop{\Z[x]}{+, \circ}$, 
which are generated by a subset of $\{1, x, x^2, x^3\}$. 
\newpage
\begin{table}[h]
\centering
\label{allPossTab}
\begin{tabular}{@{} c c c c | l | l @{}}

  $a_0$  & $a_1$ & $a_2$ & $a_3$ & \begin{tabular}[c]{@{}l@{}} $p(x) = \sum_{i=0}^n c_i x^i \in \Z[x]$ is an \\ element of the nearring generated by \\ $\{a_ix^i \mid i\in \{0, \ldots , 3\} \}$  iff for all $i \in \N_0$:   \end{tabular}                        	                                & Content of \\
\hline
\hline
  0  &  0  &   0   &   0   & $c_i = 0$                                           																				& \\
  1  &  0  &   0   &   0   & $i>0 \Rightarrow c_i = 0 $                        																						& \\
  0  &  1  &   0   &   0   & $c_0=0$ and $(i>1 \Rightarrow c_i =0)$                          																			& \\
  1  &  1  &   0   &   0   & $i>1 \Rightarrow c_i =0 $           																								&  \\
\hline
  0  &  0  &   1   &   0   & \begin{tabular}[c]{@{}l@{}} $c_0 = 0$, $c_{2i+1} =0$, \\    $ i>0 \Rightarrow 2^{s_2(i) - 1} \mid c_{2i}$  \end{tabular}                                             							& Theorem \ref{thm:x2}\\
\hline 
  1  &  0  &   1   &   0   & \begin{tabular}[c]{@{}l@{}} $c_{2i+1} = 0$, \\  $i > 0 \Rightarrow 2^{s_2(i) - 1} \mid c_{2i}$     \end{tabular}                                                              							& Theorem \ref{thm:1x2}\\
\hline
  0  &  1  &   1   &   0   &  \begin{tabular}[c]{@{}l@{}} $c_{0} = 0$, \\    $i > 0 \Rightarrow 2^{s_2(i) - 1} \mid c_{i}$     \end{tabular}                                                                							& Theorem 1.1 of~\cite{Ai:GPUF} \\
\hline 
  1  &  1  &   1   &   0   &  $i > 0 \Rightarrow 2^{s_2(i) - 1} \mid c_{i}$                                                               															& Theorem 1.2 of~\cite{Ai:GPUF} \\
\hline
  0  &  0  &   0   &   1   & \begin{tabular}[c]{@{}l@{}} $i \not\in \eqclass{6}{3}\Rightarrow c_i = 0$, \\ $i\in \eqclass{6}{3} \Rightarrow 3^{ \frac{s_3(i-1)}{2} } \mid c_{i}$, \\ $2 \mid \sum_{j \in A} c_j $ and $2 \mid \sum_{j \in B} c_j  $ \end{tabular}       & Theorem \ref{thm:x3}\\
\hline
  1  &  0  &   0   &   1   &  \begin{tabular}[c]{@{}l@{}} $i \not\in \eqclass{3}{0}\Rightarrow c_i = 0$, \\ $i\in \eqclass{3}{0}\Rightarrow 3^{ \lfloor \frac{s_3(i)}{2} \rfloor } \mid c_{i}$ \end{tabular}                                                        	& Theorem \ref{thm:1x3}\\
\hline
  0  &  1  &   0   &   1   & \begin{tabular}[c]{@{}l@{}} $i \not\in \eqclass{2}{1}\Rightarrow c_{2i+1} = 0$, \\ $i\in \eqclass{2}{1}\Rightarrow 3^{ \frac{s_3(i-1)}{2} } \mid c_{i}$, \\ $2 \mid \sum_{j \in C} c_j $ and $2 \mid \sum_{j \in D} c_j  $ \end{tabular}  & Theorem \ref{thm:xx3}\\
\hline
  1  &  1  &   0   &   1   &  $3^{ \lfloor \frac{s_3(i)}{2} \rfloor } \mid c_{i}$                                                           													& Theorem 1.3 of~\cite{Ai:GPUF}\\
\hline
  0  &  0  &   1   &   1   &   $c_0 = 0$, $c_1 =0$, $2 \mid c_5$                                                             														& Theorem \ref{thm:x2x3}\\
  1  &  0  &   1   &   1   &   $c_1 = 0 $, $2 \mid c_5$                                                            																& Theorem \ref{thm:1x2x3}\\
  0  &  1  &   1   &   1   &  $c_0 =0$                                                             																		& Theorem \ref{thm:xx2x3}\\
  1  &  1  &   1   &   1   &      $0=0$                                     																					& Theorem 4.3 of~\cite{Ai:GPUF}\\
\end{tabular}

\caption{All possible subnearrings generated by $1$, $x$, $x^2$ and $x^3$}
\end{table}

\subsection{The subnearrings generated by $ \{x^2, x^3 \}$ and $\{x, x^2, x^3 \}$ }

\begin{thm} \label{thm:x2x3}
A polynomial $p = \sum_{i=0}^n c_ix^i \in \Z[x]$ lies in the subnearring of $\algop{\Z[x]}{+, \circ} $ that is generated 
by $\{ x^2, x^3 \}$ if and only if $c_0 = 0, c_1 = 0$, and $c_5 \equiv_2 0$.
\end{thm}

\begin{proof}
 We want to use Lemma \ref{lem:genLem1}.
 Let $F := \{x^2, x^3\}$ and $$M := \{\sum_{i=0}^n c_ix^i \mid c_0=0, c_1 = 0, 2 \text{ divides } c_5 \}.$$
 The conditions $F \subseteq M$ and the closedness of $M$ under $+, -$ are obvious.
 Now we check: $F \circ M \subseteq M$.
 Let $p =  \sum_{i=0}^n c_ix^i \in M$.
 We have to show that $x^2 \circ p$ and $x^3 \circ p $ lie in $M$.
 Let $a_1, \ldots, a_{2n} \in \Z$ such that 
 $ x^2 \circ p = \sum_{i=0}^{2n} a_ix^i $ 
 and
 $b_1, \ldots, b_{3n} \in \Z$ such that  
  $ x^3 \circ p = \sum_{i=0}^{3n} b_ix^i. $ 
 Now it is sufficient to show that $a_5 \equiv_2 0$ and $ b_5 \equiv_2$.
 Let $j_1, j_2 \in \N\setminus\{ 1 \}$ such that $j_1+j_2 = 5$.
 5 is odd and thus $a_5 = 2c_{j_1+j_2}.$
 Therefore, $a_5 \equiv_2 0$.
 If $j_1, j_2, j_3 \in \N \setminus\{1\}$ then $j_1+j_2+j_3 \neq 5.$
 Therefore, $b_5 = 0$.
 
 What is left to show is that $M \subseteq \langle F \rangle$.
 Since $x^2 \in \langle F \rangle$, $x^2 \circ x^2 = x^4 \in \langle F \rangle$.
 We know that $x^3 \in \langle F \rangle$, and thus
 \begin{equation*}
  x^2 \circ (x^2 + x^3) - x^4 - x^2 \circ x^3  = 2x^5 \in \langle F \rangle.
 \end{equation*}
 Now we show by induction that for all $i \geq 6$, $x^i \in \langle F \rangle$.
 For $i = 6$, $x^2 \circ x^3 =x^6 \in \langle F \rangle$.
 For $ i = 8$, $x^2 \circ x^2 \circ x^2 = x^8 \in \langle F \rangle$.
 For $i=7$,
 \begin{equation*}
 \begin{aligned}
 x^2 \circ (x^2 +2 x^5) - x^4 &- x^2 \circ (2x^5)  \\ &- (x^3 \circ (x^2 +x^3) -x^6 - x^3 \circ x^3 -3x^8) = x^7 \in \langle F \rangle.
 \end{aligned}
 \end{equation*}
 For $i=9$, $x^3 \circ x^3 = x^9 \in \langle F \rangle. $
 For $i=10$, 
 \begin{equation*}
  \begin{aligned}
   x^2 \circ (2x^5) - ( x^3 \circ (x^2 + x^4) -x^4 -3 x^8 - x^3 \circ x^4 ) = x^{10} \in \langle F \rangle.
  \end{aligned}
 \end{equation*}
 For $i=11$,
 \begin{equation*}
  \begin{aligned}
   x^3 \circ (x^3 +x^4) - x^9 &-3x^{10}-x^3 \circ x^4 \\&-(x^2 \circ (x^7+x^4) -x^2 \circ x^7 -x^8) = x^{11} \in \langle F \rangle.
  \end{aligned}
 \end{equation*}
 For $i = 12$, $ x^3 \circ x^4 =x^{12} \in \langle F \rangle$.
 For $ i =13$, we first show that $3x^{17} \in \langle F \rangle$.
 This is true since
 $$ x^3 \circ (x^4+x^9) -x^{12} -3 \cdot (x^2 \circ x^{11}) -x^3 \circ x^9 = 3x^{17} \in \langle F \rangle.$$
 Now we have that
 \begin{equation*}
  \begin{aligned}
   x^3 \circ (x^3 + x^7) -x^9 &-3x^{17} -x^3 \circ x^7 \\&-(x^2 \circ (x^2 +x^{11}) -x^4 -x^2 \circ x^{11}) = x^{13} \in \langle F \rangle.  
  \end{aligned}
 \end{equation*}

 For the induction step we let $ i \geq 14$.
 If $ i$ is even, let $j_2 = 2 $, otherwise $j_2 =3$.
 Let $j_1 := \frac{i-j_2}{2}$, then $i = 2j_1 +j_2$.
 Since $j_1 \geq 6$ the induction hypothesis yields that $x^{j_1}$, $x^{2j_1}$, $x^{j_2}$ and $x^{2j_2+j_1}$ lie in $\langle F \rangle$.
 Therefore we get
 \begin{equation}\label{eq:x2x3eq1}
 \begin{aligned}
   x^3 \circ (x^{j_1} +x^{j_2}) - x^3 \circ x^{j_2} -x^3 \circ x^{j_1} -3 x^{2j_2+j_1} = 3x^{2j_1+j_2} \in \langle F \rangle
 \end{aligned}
 \end{equation}
 and also
  \begin{equation}\label{eq:x2x3eq2}
 \begin{aligned}
   x^2 \circ (x^{2j_1} +x^{j_2}) - x^2 \circ x^{2j_1} -x^2 \circ x^{j_2} = 2x^{2j_1+j_2} \in \langle F \rangle.
 \end{aligned}
 \end{equation}

 Now we subtract \eqref{eq:x2x3eq2} from \eqref{eq:x2x3eq1} and get
 $$ x^{2j_1 +j_2} \in \langle F \rangle, $$
 which concludes the proof.
\end{proof}

\begin{thm} \label{thm:xx2x3}
A polynomial $p = \sum_{i=0}^n c_ix^i \in \Z[x]$ lies in the subnearring of $\algop{\Z[x]}{+, \circ} $ that is generated 
by $\{x, x^2, x^3 \}$ if and only if $c_0 = 0$.
\end{thm}

\begin{proof}
 We have 
 \begin{equation*}
  \begin{aligned}   
x^3 \circ (x+x^2) - x^3 &- x^3 \circ x^2 - 3 \cdot x^2 \circ x^2 
\\&- (x^2 \circ (x^2 +x^3) -x^4 - x^2 \circ x^3) = x^5 \in  \langle \{x, x^2, x^3 \} \rangle.
 \end{aligned}
 \end{equation*}
 We are done, since $\{x^2, x^3\}$ is a subset of $\{x, x^2, x^3\}$ and by Theorem \ref{thm:x2x3} we know that
 for all $j \in \N\setminus\{1,5\}$, we have $x^j \in \langle \{ x^2, x^3\} \rangle$.  
\end{proof}

\subsection{The subnearring generated by $\{1, x^2, x^3 \}$}

\begin{thm} \label{thm:1x2x3}
A polynomial $p = \sum_{i=0}^n c_ix^i \in \Z[x]$ lies in the subnearring of $\algop{\Z[x]}{+, \circ} $ that is generated 
by $\{1, x^2, x^3 \}$ if and only if $ c_1 = 0$ and $2$ divides $c_5$.
\end{thm}

\begin{proof}
 We want to use Lemma \ref{lem:genLem1}.
 Let $F := \{1, x^2, x^3\}$ and $$M := \{\sum_{i=0}^n c_ix^i \mid  c_1 = 0  \text{ and } 2 \text{ divides } c_5 \}.$$
 The conditions $F \subseteq M$ and the closedness of $M$ under $+, -$ are obvious.
 Now we show $M \subseteq \langle F \rangle$. 
 By Theorem \ref{thm:x2x3} we have
 $\{\sum_{i=1}^n c_ix^i \mid  c_1 = 0  \text{ and } 2 \text{ divides } c_5 \} \subseteq \langle \{ x^2, x^3 \} \rangle$.
 We are done, since $1 \in F$ and $\langle \{ x^2, x^3 \} \rangle \subseteq \langle \{1, x^2, x^3 \} \rangle.$ 
 
 Now we show $F \circ M \subseteq M$.
 Let $p =  \sum_{i=0}^n c_ix^i \in M$.
 We have to show that $x^2 \circ p$ and $x^3 \circ p $ lie in $M$.
 Let $a_1, \ldots, a_{2n} \in \Z$ such that 
 $ x^2 \circ p = \sum_{i=0}^{2n} a_ix^i $ 
 and
 $b_1, \ldots, b_{3n} \in \Z$ such that  
  $ x^3 \circ p = \sum_{i=0}^{3n} b_ix^i. $ 
 Now it is sufficient to show that $a_1=0$, $b_1 =0$, $a_5 \equiv_2 0$ and $ b_5 \equiv_2 0$.
 
 Since $p \in M$, $c_1 = 0$.
 Thus, $a_1 = 2 c_0 c_1 = 0$.
 We have $a_5 = 2 c_1 c_4 + 2 c_2 c_3 + 2 c_0 c_5$.
 Therefore $a_5 \equiv_2 0$.
 Furthermore, $b_1 = 3 c_0^2 c_1 = 0 $ since $c_1 = 0.$ 
 Since $c_1 = 0$,  $b_5 = 6 c_0 c_2 c_3$ and this is divisible by 2.
 \end{proof}

\subsection{The subnearrings generated by $\{ x^2 \}$ and $\{ 1, x^2\}$} \label{sec:x2}

\begin{thm} \label{thm:x2}
A polynomial $p = \sum_{i=0}^n c_ix^i \in \Z[x]$ lies in the subnearring of $\algop{\Z[x]}{+, \circ} $ that is generated 
by $\{ x^2 \}$ if and only if $c_0=0$ and for all $i \in \N$, $2^{s_2(i) - 1}$ divides $c_{2i}$, and $c_{2i-1} =0$.
\end{thm}

\begin{proof} 
``$\Rightarrow$'':  
By Lemma \ref{lem:genLem3} we know that there exists $q = \sum_{i=0}^{\lfloor \frac{n}{2} \rfloor} \tilde{c}_ix^i  \in \langle \{x, x^2 \} \rangle$ 
such that $p = q(x^2)$.
Hence $p$ does not contain monomials of odd degree.
By Theorem 1.1 of~\cite{Ai:GPUF}, 
$\tilde{c}_0 = 0$ and for all $i \in \N, 2^{s_2(i) - 1}$ divides $\tilde{c}_{i}=c_{2i}$.
``$\Leftarrow$'': 
Let $q \in \Z[x]$ such that $q \circ x^2 = p$.
Then $q$ lies in $\langle \{x, x^2 \} \rangle$ by Theorem 1.1 of~\cite{Ai:GPUF}, because $s_2(2i) = s_2(i)$ for all $i \in \N_0$.
Therefore Lemma \ref{lem:genLem3} yields $p \in \langle \{x^2\} \rangle$.
\end{proof}

\begin{thm} \label{thm:1x2}
A polynomial $p = \sum_{i=0}^n c_ix^i \in \Z[x]$ lies in the subnearring of $\algop{\Z[x]}{+, \circ} $ that is generated 
by $\{1, x^2 \}$ if and only if for all $i \in \N$, $2^{s_2(i) - 1}$ divides $c_{2i}$, and $c_{2i-1} =0$.
\end{thm}

\begin{proof}
``$\Rightarrow$'':  
By Lemma \ref{lem:genLem2} we know that there exists $q \in \langle \{1, x, x^2 \} \rangle$, 
such that $p = q(x^2)$.
By Theorem 1.2 of~\cite{Ai:GPUF} it follows that
$c_{2i}$ is a multiple of $2^{s_2(i) - 1}$ for all $i \in \N_0$.
``$\Leftarrow$'': 
Let $q \in \Z[x]$ such that $q \circ x^2 = p$.
Then $q$ lies in $\langle \{1, x, x^2 \} \rangle$ by Theorem 1.2 of~\cite{Ai:GPUF} because $s_2(2i) = s_2(i)$ for all $i \in \N_0$.
Therefore Lemma \ref{lem:genLem2} yields $p \in \langle \{1, x^2\} \rangle$.
\end{proof}

\subsection{The subnearrings generated by $\{ x^3 \}$, $\{ x, x^3\}$ and $\{1, x^3\}$} \label{sec:x3}

The following definitions for $M$, $a(i)$ and $b(i)$ are just for this section. 
Let $A, B$ be the subsets of $\N_0$ defined at the beginning of section \ref{sec:Examples}.

\begin{de}
Let $M$ be the subset of $\Z[x]$ defined by
\begin{equation*}
 \begin{aligned}
  M := \bigg\{ \sum_{i=0}^n c_ix^i  \, \big\vert & \, n \in \N_0 , \, 
\forall i \in \N_0 \setminus \eqclass{6}{3}: c_i = 0, \\ & \forall i \in \eqclass{6}{3}:3 ^{\frac{s_3(i)-1}{2}} \mid c_i, \;
 2 \mid \SUM{j \in A}{}{c_j}, \, 2 \mid \SUM{ j \in B}{}{c_j} \bigg\} .
 \end{aligned}
\end{equation*}
\end{de}

\begin{de}
 Let $a(i)$ and $b(i)$ defined by
 $$ a(i):= \begin{cases} 
	      0 & \text{ if } i \not\in A, \\
	      3 & \text{ if } i \in A    
           \end{cases} $$         
and
$$ b(i) := \begin{cases} 
	      0 & \text{ if } i \not\in B, \\
	      3 & \text{ if } i \in B. 
           \end{cases}$$
\end{de}

\begin{lem} \label{lem:x3elemofF}
 Let $F := \{ x^3 \}$, $a \in \N$ and $i \in \N_0$. Then we have:
 \begin{enumerate}
  \item $x^{3^a} \in \langle F \rangle $. \label{lem:x3elemofF1}
  \item  If $p(x) \in \langle F \rangle $ then 
  $3p(x)x^{2 \cdot 3^a} + 3p(x)^2x^{3^a}$, $6p(x)x^{2 \cdot 3^a}$ and $6p(x)^2x^{3^a}$ lie in $\langle F \rangle$. \label{lem:x3elemofF2}
  \item The polynomials $3x^{21}+3x^{15} $, $3x^{57}+3x^{33}$, $3x^{45}+3x^{63}$, $6x^{15}$, $6x^{33}$, $6x^{45}$, $6x^{63}$ lie in $\langle F \rangle $. \label{lem:x3elemofF3}
 \end{enumerate}
\end{lem}

\begin{proof}
The first item can be proved by induction on $a$.
For the second item we observe that 
\begin{equation*}
\begin{aligned}
x^3 \circ (p(x) + x^{3^a} ) - x^3 \circ (p(x)) - x^{3^{a+1}} &= p(x)^3 + 3p(x)x^{2 \cdot 3^a} + 3p(x)^2x^{3^a} + x^{3^{a+1}} 
\\ & \qquad- p(x)^3 -x^{3^{a+1}} \\ &= 3p(x)x^{2 \cdot 3^a} + 3p(x)^2x^{3^a}. 
\end{aligned}
\end{equation*}
If $p(x) \in \langle F \rangle$ we also have $-p(x) \in \langle F \rangle$.
Thus, $$ (3p(x)x^{2 \cdot 3^a} + 3p(x)^2x^{3^a}) - ( - 3p(x)x^{2 \cdot 3^a} + 3p(x)^2x^{3^a}) = 6p(x)x^{2 \cdot 3^a}\in \langle F \rangle.$$
Furthermore, $$3p(x)x^{2 \cdot 3^a} + 3p(x)^2x^{3^a} + (-3p(x)x^{2 \cdot 3^a} + 3p(x)^2x^{3^a}) = 6p(x)^2x^{ 3^a} \in \langle F \rangle.$$
For item \eqref{lem:x3elemofF3} we observe that setting $p(x):= x^3$, $a := 2$ in item \eqref{lem:x3elemofF2} yields that
$3x^{21}+3x^{15}$ and $6x^{15}$ lie in $ \langle F \rangle$.
Setting $p(x):= x^3$, $a := 3$, item \eqref{lem:x3elemofF2} yields that $3x^{57}+3x^{33}$ and $6x^{33}$ lie in $\langle F \rangle$.
Setting $p(x):= x^9$, $a := 3$, item \eqref{lem:x3elemofF2} yields that $3x^{63}+3x^{45}$, $6x^{45}$ and $6x^{63}$ lie in $\langle F \rangle$. 
\end{proof}

\begin{lem}\label{lem:help3lxi}
Let $F := \{x^3 \}$, $i, j \in \eqclass{6}{3}$, $l_1, l_2, l_3 \in \N_0$, and $o_1, o_2, o_3, o_4 \in \{0, 3\}$
such that
\begin{equation}\label{eq:helplem1}
 3^{l_1}x^j +3^{l_2}x^i + o_1 x^{15} + o_2  x^{33}
\end{equation}
and
\begin{equation}\label{eq:helplem2}
 3^{l_3} x^j + o_3  x^{15} + o_4 x^{33}
\end{equation}
lie in $\langle F \rangle.$
We define 
 $ o_5 := \begin{cases} 
	      0 & \text{ if }  o_1= o_3,  \\
	      3 & \text{ if } o_1 \neq o_3 
           \end{cases} $  
and 
 $ o_6 := \begin{cases} 
	      0 & \text{ if }  o_2 = o_4,  \\
	      3 & \text{ if } o_2 \neq o_4
           \end{cases}.$  
Then there exists $t \in \N$ such that
$$3^tx^i + o_5 x^{15} + o_6 x^{33} \in \langle F \rangle.$$
\end{lem}

\begin{proof}
A linear combination of \eqref{eq:helplem1} and \eqref{eq:helplem2} yields that there exist $t \in \N$, $c_1, c_2\in \N$ such that 
$$3^tx^i + (3^{c_1} o_1 - 3^{c_2} o_3 ) x^{15} + (3^{c_1} o_2 - 3^{c_2} o_4) x^{33} \in \langle F \rangle.$$
By item \ref{lem:x3elemofF3} of Lemma \ref{lem:x3elemofF}, 
$6x^{15}$ and $6x^{33}$ lie in $\langle F \rangle$.
Since $3^{c_1} o_1 - 3^{c_2} o_3  \equiv_6 o_5$ and $3^{c_1} o_2 - 3^{c_2} o_4 \equiv_6 o_6$ we have
$$3^tx^i + o_5 x^{15} + o_6 x^{33} \in \langle F \rangle.$$
\end{proof}

\begin{lem}\label{lem:3lxiax15bx33inF}
Let $F := \{x^3 \}$. 
For all $ i \in \eqclass{6}{3} $ there exists $l \in \N_0$ such that $$3^{l} x^i + a(i)x^{15} + b(i)x^{33} \in \langle F \rangle .$$
\end{lem}

\begin{proof}
For the induction basis we used Mathematica \cite{Pr:Mathe}. The program is openly available on \cite{Url:NearringGen}. 
We get the following list of polynomials which lie in $\langle F \rangle$ as a part of the output.
\begin{equation} \label{eq:indlist}
 \begin{aligned}
    \{ &0, 0, 0, x^3, 0, 0, 0, 0, 0, x^9, 0, 0, 0, 0, 0, 6x^{15}, 
0, 0, 0,  0, 0, \\ & 3x^{15} + 3x^{21}, 0, 0, 0, 0, 0, x^{27}, 0, 0, 
0, 0, 0,  6x^{33}, 0,  0, 0, \\ & 0, 0, 3x^{15} + 3x^{39}, 0, 0, 0, 0,
0,  3x^{15} + 3x^{33} + 3x^{45},  0, 0, \\ &0, 0, 0, 3x^{33} + 9x^{51},
0, 0, 0, 0, 0, 3x^{33} + 3x^{57},  0, 0, 0,  0, \\ & 0, 3x^{15} + 3x^{33} + 3x^{63}, 
0, 0, 0, 0, 0,  3x^{15} + 9x^{69}, 0, 0,  0, \\ & 0, 0, 3x^{33} + 9x^{75}, \ldots \} .
 \end{aligned}
\end{equation}
For the induction step, we let $i \geq 81$.
We consider the cases $i \in \eqclass{12}{9}$ and $i \in \eqclass{12}{3}$. \\
\textit{Case:} $i \in \eqclass{12}{9}$.
Let $ r := \frac{i-3}{2} $. 
Then $ r \in \eqclass{6}{3}$ and $i=2 r +3$.
Since $i \geq 81$, $r \geq 39$.
By the induction hypothesis there exists a $l_1 \in \N_0$ such that
$$q(x) := 3^{l_1}x^r + a(r) x^{15} + b(r) x^{33} \in \langle F \rangle .$$
By item \eqref{lem:x3elemofF2} of Lemma \ref{lem:x3elemofF} we have $3q(x)x^{2 \cdot 3} + 3q(x)^2x^{3} \in \langle F \rangle$.
Therefore, we have that
\begin{equation} \label{eq:firsteq3^l}
  \begin{aligned}
    3^{l_1+1} x^{r + 6 } &+3a(r)  x^{21 } +  3b(r)  x^{39}   +  3 a(r)^2 x^{33} +  2 \cdot 3a(r)b(r) x^{51} + 3 b(r)^2 x^{69} \\ &+ 3 \cdot 3^{2l_1} x^{2 r+3} + 
    2 \cdot 3^{l_1+1} a(r) x^{18 + r} + 2  \cdot 3^{l_1+1} b(r) x^{36 + r} \in \langle F \rangle. 
 \end{aligned}
\end{equation}
By \eqref{eq:indlist} we have $3 x^{33} + 9 x^{51} \in \langle F \rangle$ and $6x^{33} \in \langle F \rangle$.
Hence $ 2 \cdot (3 x^{33} + 9 x^{51} ) - 6x^{33} = 18 x^{51} \in \langle F \rangle$.
Since $b(r) = 0 \text{ or } b(r) =3$ this implies $6 a(r) b(r) x^{51} \in \langle F \rangle$.
By \eqref{eq:indlist},  $6x^{15}\in\langle F \rangle$, thus $2 \cdot b(r)x^{15}$ and $2\cdot b(r)^2x^{15} $ lie in $\langle F \rangle$.
By \eqref{eq:indlist} we also have that $3x^{15} + 3x^{39}$ and $9 x^{69} +3x^{15}$ lie in $\langle F \rangle$.
Since  $ 2\cdot  (9x^{69} +3x^{15}) -6x^{15} =18x^{69} \in \langle F \rangle$ we have $6b(r)^2 x^{69} \in \langle F \rangle.$ 
Hence,
\begin{equation*}
 \begin{aligned}
  b(r) &\cdot (3x^{15} + 3x^{39})  + b(r)^2 \cdot (9 x^{69} +3x^{15}) - 2 \cdot b(r)x^{15}  \\  &- 2 \cdot b(r)^2x^{15} -6b(r)^2 x^{69}  
  = 3b(r)x^{39}+ 3 b(r)^2 x^{69} + (b(r) +b(r)^2)x^{15} \in \langle F \rangle.
 \end{aligned}
\end{equation*}
Since  $b(r) = 0 \text{ or } b(r) =3$ and $6x^{15} \in \langle F \rangle$ we get $3b(r)x^{39}+ 3 b(r)^2 x^{69} \in \langle F \rangle$.
We also know by \eqref{eq:indlist} that $3x^{15}+3x^{21}\in \langle F \rangle $ and hence $a(r)\cdot(3x^{15}+3x^{21}) \in \langle F \rangle$.
Therefore, we get by subtracting $3b(r)x^{39}+ 3 b(r)^2 x^{69}$, $a(r)\cdot(3x^{15}+3x^{21})$ and $6 a(r) b(r) x^{51} $                       
from \eqref{eq:firsteq3^l},
\begin{equation*} 
\begin{aligned}
 3^{l_1+1}x^{r + 6 } &-3a(r)x^{15} + 3a(r)^2 x^{33} +3^{2l_1+1}x^{2r+3} \\ &+  2 \cdot 3^{l_1+1} a(r) x^{18 + r} + 2  \cdot 3^{l_1+1} b(r) x^{36 + r} \in \langle F \rangle . 
 \end{aligned}
\end{equation*}
Since $6x^{15}$ and $6x^{33}$ lie in $\langle F \rangle $ and $a(r) = 0$ or $a(r) = 3$ we also get
\begin{equation} \label{eq:firsteq3^l2}
\begin{aligned}
  3^{l_1+1}x^{r + 6 }  &+ a(r)x^{15} + a(r) x^{33} +3^{2l_1+1}x^{2r+3} \\ &+ 2 \cdot 3^{l_1+1} a(r) x^{18 + r} + 2  \cdot 3^{l_1+1} b(r) x^{36 + r} \in \langle F \rangle . 
  \end{aligned}
\end{equation}
Since $36 +r < 2 r + 3$ we know by the induction hypothesis that there exist $l_2, l_3 \in \N_0$ such that
$3^{l_2}x^{r+18} + a(r+18) x^{15} + b(r+18)x^{33}$ and $3^{l_3}x^{r+36} + a(r+36) x^{15} + b(r+36)x^{33}$ lie in $\langle F \rangle$.
Therefore,y \eqref{eq:indlist} 
\begin{equation*}\label{eq:firsteq3^l3}
 \begin{aligned}
  2 \cdot(3^{l_2}x^{r+18} + a(r+18) x^{15} + b(r+18)x^{33}) &- 2 \cdot a(r+18) x^{15} \\&- 2 \cdot b(r+18)x^{33}   =2 \cdot 3^{l_2}x^{r+18} \in \langle F \rangle
 \end{aligned}
\end{equation*}
and
\begin{equation*}\label{eq:firsteq3^l4}
 \begin{aligned}
  2 &\cdot(3^{l_3}x^{r+36} + a(r+36) x^{15} + b(r+36)x^{33}) - 2 \cdot a(r+36) x^{15} \\&- 2 \cdot b(r+36)x^{33} =  2 \cdot 3^{l_3}x^{r+36} \in \langle F \rangle.
 \end{aligned}
\end{equation*}
Let $L_1 := \lcm(3^{l_1+1}, 3^{l_2}, 3^{l_3})$.
Then there exists a $k_1 \in \N$ such that 
$3^{k_1}\cdot 3^{l_1+1} \geq L_1 $,   $3^{k_1}\cdot3^{l_2}\geq L_1$ and $3^{k_1}\cdot 3^{l_3}\geq L_1$.
By multiplying \eqref{eq:firsteq3^l2} with $3^{k_1}$ and subtracting $a(r) \cdot 3^{k_1} \cdot  2 \cdot 3^{l_1+1}x^{r+18} \in \langle F \rangle $ and  
$ b(r) \cdot 3^{k_1} \cdot 2 \cdot 3^{l_1+1}x^{r+36} \in \langle F \rangle$ we get
\begin{equation}\label{eq:1:3^k_1}
 3^{l_1+1+k_1}x^{r + 6 }  + 3^{k_1}a(r)x^{15} + 3^{k_1}a(r) x^{33} +3^{2l_1+1+k_1}x^{2r+3} \in \langle F \rangle .
\end{equation}
For all $k \in \N$ it holds that $3^ka(r) \equiv_6 a(r)$.
Since $6x^{15}$ and $6x^{33}$ lie in $\langle F \rangle$, we get 
\begin{equation} \label{eq:impeq1}
 3^{l_1+k_1+1}x^{r + 6 }  + 3^{2l_1 +k_1+1}x^{2r+3} + a(r)x^{15} + a(r) x^{33} \in \langle F \rangle . 
\end{equation}
Now we have to distinguish between some cases for $r$ according to their remainder modulo 72. 
Therefore, we consider the following table.
We take as an example the first row of the table and read it in the following way.
Let $r \in \eqclass{72}{9}$. 
Then $a(r)=0$ by the definition of $a(r)$.
We get that $r+6 \in \eqclass{72}{15}$ and therefore the induction hypothesis yields a $l_4 \in \N_0 $
such that $3^{l_4}x^{r+6} + 3x^{15} \in \langle F \rangle$.   
Now we use $\tilde{x}$ for the variable $x$ in the statement of Lemma \ref{lem:help3lxi}.
We set $\tilde{i} := 2r+3$, $\tilde{j}:= r+6$,
$\tilde{l_1}:= l_1+k_1+1$, $\tilde{l_2}:= 2l_1+k_1+1$, $\tilde{l_3}:=l_4$,
$\tilde{o_1}:= a(r) = 0$, $\tilde{o_2}:= a(r) = 0$,
$\tilde{o_3}:= 3$ and $\tilde{o_4}:= 0$.
By Lemma \ref{lem:help3lxi} we have that there exists a $\tilde{t} \in \N_0$ such that 
$3^{\tilde{t}}x^{2r +3} + 3x^{15} \in \langle F \rangle$ . 
Now we set $t:=\tilde{t}$ and thus $3^{t}x^{2r +3} + 3x^{15} \in \langle F \rangle$.
We have $i=2r+3 \in \eqclass{72}{21}$ and this satisfies our statement, since $a(i)=3$ and $b(i) = 0$. 
The other 11 cases work in a similar way.

\begin{table}[h]
\centering
\footnotesize
\setstretch{1.5}
\setlength{\tabcolsep}{2pt}
\label{firsttable}
\begin{tabular}{@{}|l||c|l|lll|lll|l|c|c|@{}}
\hline  
 \setstretch{0.7} \begin{tabular}[l]{@{}l@{}}  Case: \\ $r \in $ \end{tabular}    & $a(r)$        &  \setstretch{0.7}  \begin{tabular}[c]{@{}l@{}} Then \\ $r+6 \in $ \end{tabular}    &   \multicolumn{3}{l|}{  \setstretch{0.7} \begin{tabular}[l]{@{}l@{}} \\ Then the induction \\ hypothesis yields a \\ $l_4 \in \N_0$ such that \\$ \ldots \in \langle F \rangle$ \\ \\ \end{tabular} }   &     \multicolumn{3}{l|}{ \setstretch{0.7}  \begin{tabular}[l]{@{}l@{}} \\ Together with \eqref{eq:impeq1}, \\ Lemma \ref{lem:help3lxi} yields a \\ $t \in \N_0$ such that \\$ \ldots \in \langle F \rangle$ \\ \\ \end{tabular}}    &  $i \in$  &  $a(i)$ &  $b(i)$ \\ 
\hline
$\eqclass{72}{9}$                                              & 0             & $\eqclass{72}{15}$                                              & $3^{l_4}x^{r+6}$  & $+ 3x^{15}$     &                                                                                                                                              & $3^{t}x^{2r +3}$     &    $+ 3x^{15}$   &                                                                                                                                                         & $\eqclass{72}{21}$  & 3 & 0\\
$\eqclass{72}{15}$                                             & 3             & $\eqclass{72}{21}$                                              & $3^{l_4}x^{r+6}$  & $+ 3 x^{15}$    &                                                                                                                                              & $3^{t}x^{2r +3}$     &                  &  $+ 3x^{33}$                                                                                                                                            & $\eqclass{72}{33}$  & 0 & 3\\
$\eqclass{72}{21}$                                             & 3             & $\eqclass{72}{27}$                                              & $3^{l_4}x^{r+6}$  &                 &                                                                                                                                              & $3^{t}x^{2r +3}$     &    $+ 3x^{15}$   &  $+ 3x^{33}$                                                                                  							  & $\eqclass{72}{45}$  & 3 & 3\\
$\eqclass{72}{27}$                                             & 0             & $\eqclass{72}{33}$                                              & $3^{l_4}x^{r+6}$  &                 & $+ 3x^{33}$                                                                                                                                  & $3^{t}x^{2r +3}$     &                  &  $+ 3x^{33}$                                                                                                                                            & $\eqclass{72}{57}$  & 0 & 3\\
$\eqclass{72}{33}$                                             & 0             & $\eqclass{72}{39}$                                              & $3^{l_4}x^{r+6}$  & $+ 3x^{15}$     &                                                                                                                                              & $3^{t}x^{2r +3}$     &    $+ 3x^{15}$   &                                                                                                                                                         & $\eqclass{72}{69}$  & 3 & 0\\
$\eqclass{72}{39}$                                             & 3             & $\eqclass{72}{45}$                                              & $3^{l_4}x^{r+6}$  & $+ 3x^{15}$     & $+ 3x^{33}$                                                                                                                                  & $3^{t}x^{2r +3}$     &                  &                                                                                                                                                         & $\eqclass{72}{9}$   & 0 & 0\\
$\eqclass{72}{45}$                                             & 3             & $\eqclass{72}{51}$                                              & $3^{l_4}x^{r+6}$  &                 & $+ 3x^{33}$                                                                                                                                  & $3^{t}x^{2r +3}$     &    $+ 3x^{15}$   &                                                                                                                                                         & $\eqclass{72}{21}$  & 3 & 0\\
$\eqclass{72}{51}$                                             & 0             & $\eqclass{72}{57}$                                              & $3^{l_4}x^{r+6}$  &                 & $+ 3x^{33}$                                                                                                                                  & $3^{t}x^{2r +3}$     &                  &   $+ 3x^{33}$                                                                                                                                           & $\eqclass{72}{33}$  & 0 & 3\\
$\eqclass{72}{57}$                                             & 0             & $\eqclass{72}{63}$                                              & $3^{l_4}x^{r+6}$  & $+ 3x^{15}$     & $+ 3x^{33}$                                                                                                                                  & $3^{t}x^{2r +3}$     &    $+ 3x^{15}$   &   $+ 3x^{33}$                                                                                                                                           & $\eqclass{72}{45}$  & 3 & 3\\
$\eqclass{72}{63}$                                             & 3             & $\eqclass{72}{69}$                                              & $3^{l_4}x^{r+6}$  & $+ 3x^{15}$     &                                                                                                                                              & $3^{t}x^{2r +3}$     &                  &   $+ 3x^{33}$                                                                                                                                           & $\eqclass{72}{57}$  & 0 & 3\\
$\eqclass{72}{69}$                                             & 3             & $\eqclass{72}{3}$                                               & $3^{l_4}x^{r+6}$  &                 & $+ 3 x^{33} $                                                                                                                                & $3^{t}x^{2r +3}$     &    $+ 3 x^{15}$  &                                                                                                                                                         & $\eqclass{72}{69}$  & 3 & 0\\
$\eqclass{72}{3}$                                              & 0             & $\eqclass{72}{9}$                                               & $3^{l_4}x^{r+6}$  &                 &                                                                                                                                              & $3^{t}x^{2r +3}$     &                  &                                                                                                                                                         & $\eqclass{72}{9}$   & 0 & 0\\
\hline
\end{tabular}

\caption{\textit{Case:} $i \in \eqclass{12}{9}$. Distinction for $r \mod 72$ in $\eqclass{6}{3}$}
\end{table}

Hence in each case, there exists a $t \in \N_0$ such that 
$$3^{t} x^i + a(i)x^{15} + b(i)x^{33} \in \langle F \rangle .$$
\textit{Case:} $ i \in \eqclass{12}{3}$.
Let $r:= \frac{i-9}{2}$.
Then $ r \in \eqclass{6}{3}$ and $ i = 2 r + 9$. 
By the induction hypothesis, there exists a $l_1 \in \N_0$ such that
$$q(x) := 3^{l_1}x^r + a(r) x^{15} + b(r) x^{33} \in \langle F \rangle .$$
By item \eqref{lem:x3elemofF2} of Lemma \ref{lem:x3elemofF} we have $3q(x)x^{2 \cdot 9} + 3q(x)^2x^{9} \in \langle F \rangle.$
Therefore, we have that
\begin{equation} \label{eq:secondeq3^l}
 \begin{aligned} 
  3^{l_1+1} x^{r + 18 } &+ 3a(r)  x^{33 } +  3b(r)  x^{51}   +  3 a(r)^2 x^{39} +  2 \cdot 3a(r) b(r) x^{57} + 3 b(r)^2 x^{75} \\ &+ 3 \cdot 3^{2l_1} x^{2 r+9} + 
 2 \cdot 3^{l_1+1} a(r) x^{24+r} + 2  \cdot 3^{l_1+1} b(r) x^{42 + r} \in \langle F \rangle. 
 \end{aligned}
\end{equation}
By \eqref{eq:indlist} we have $3 x^{33} + 3 x^{57} \in \langle F \rangle$ and $6x^{33} \in \langle F \rangle$.
Hence $ 2 \cdot (3 x^{33} + 3 x^{57} ) - 6x^{33} = 6 x^{57} \in \langle F \rangle$.
Thus $6 a(r) b(r) x^{57} \in \langle F \rangle$.
By \eqref{eq:indlist} we have that $3x^{33} + 9x^{75}$ and $3 x^{33} +9x^{51}$ lie in $\langle F \rangle$.
Since $b(r) \in \{0, 3\}$ and $6x^{33} \in \langle F \rangle$, $b(r)(3x^{33} + 3x^{51})$ and $b(r)^2(3 x^{33} +3 x^{75})$ lie in $\langle F \rangle$.
Therefore, we get that
\begin{equation*}
 \begin{aligned}
   b(r)(3x^{33} + 3x^{51})  +  b(r)^2(3 x^{33} +3 x^{75}) = 3b(r)x^{51} &+ 3 b(r)^2 x^{75} \\&+ (3b(r) +3b(r)^2)x^{33} \in \langle F \rangle.
 \end{aligned}
\end{equation*}
Since $6x^{33} \in \langle F \rangle$ and $ 3b(r) +3b(r)^2 \equiv_6 0 $, we get $3b(r)x^{51}+ 3 b(r)^2 x^{75} \in \langle F \rangle$.
We also know by \eqref{eq:indlist} that $3x^{15}+3x^{39} \in \langle F \rangle $ and hence $a(r)^2 \cdot (3x^{15}+3x^{39}) \in \langle F \rangle$.
Therefore, we get by subtracting $ 3b(r)x^{51}+ 3 b(r)^2 x^{75} $, $a(r)^2 \cdot (3x^{15}+3x^{39})$ and $6 a(r) b(r) x^{57}$
from \eqref{eq:secondeq3^l}, 
\begin{equation*} 
\begin{aligned}
  3^{l_1+1}x^{r + 18 } &- 3a(r)^2x^{15} + 3a(r)  x^{33} +3^{2l_1+1}x^{2r+9} \\ &+  2 \cdot 3^{l_1+1} a(r) x^{24+r} + 2  \cdot 3^{l_1+1} b(r) x^{42 + r} \in \langle F \rangle. 
 \end{aligned}
\end{equation*}
Since $6x^{15}$ and $6x^{33}$ lie in $\langle F \rangle $ and $a(r) = 0$ or $a(r) = 3$ we also get
\begin{equation} \label{eq:secondeq3^l2}
\begin{aligned}
  3^{l_1+1}x^{r + 18 }  &+ a(r)x^{15} + a(r) x^{33} +3^{2l_1+1}x^{2r+9} \\ &+ 2 \cdot 3^{l_1+1} a(r) x^{24+r} + 2  \cdot 3^{l_1+1} b(r) x^{42 + r} \in \langle F \rangle . 
  \end{aligned}
\end{equation}
Since $42 +r < 2 r + 9$ we know by the induction hypothesis that there exist $l_2, l_3 \in \N_0$ such that
$3^{l_2}x^{r+24} + a(r+24) x^{15} + b(r+24)x^{33}$ and $3^{l_3}x^{r+42} + a(r+42) x^{15} + b(r+42)x^{33}$ lie in $\langle F \rangle$.
Therefore, 
\begin{equation*}\label{eq:secondeq3^l3}
 \begin{aligned}
  2 \cdot(3^{l_2}x^{r+24} + a(r+24) x^{15} + b(r+24)x^{33}) &- 2 \cdot a(r+24) x^{15} \\&- 2 \cdot b(r+24)x^{33}   =2 \cdot 3^{l_2}x^{r+24} \in \langle F \rangle
 \end{aligned}
\end{equation*}
and
\begin{equation*}\label{eq:secondeq3^l4}
 \begin{aligned}
  2 \cdot(3^{l_3}x^{r+42} &+ a(r+42) x^{15} + b(r+42)x^{33}) - 2 \cdot a(r+42) x^{15} \\&- 2 \cdot b(r+42)x^{33} =  2 \cdot 3^{l_3}x^{r+42} \in \langle F \rangle.
 \end{aligned}
\end{equation*}
Let $L_2 := \lcm(3^{l_1+1}, 3^{l_2}, 3^{l_3})$.
Then there exists a $k_2 \in \N$ such that 
$3^{k_2}\cdot 3^{l_1+1} \geq L_2 $,   $3^{k_2}\cdot3^{l_2}\geq L_2$ and $3^{k_2}\cdot 3^{l_3}\geq L_2$.
By multiplying \eqref{eq:secondeq3^l2} with $3^{k_2}$ and subtracting $a(r) \cdot 3^{k_2} \cdot  2 \cdot 3^{l_1+1}x^{r+24} \in \langle F \rangle $ and  
$ b(r) \cdot 3^{k_2} \cdot 2 \cdot 3^{l_1+1}x^{r+42} \in \langle F \rangle$ we get
\begin{equation}\label{eq:1:3^k_2}
 3^{l_1+1+k_2}x^{r + 18 }  + 3^{k_2}a(r)x^{15} + 3^{k_2}a(r) x^{33} +3^{2l_1+1+k_2}x^{2r+9} \in \langle F \rangle .
\end{equation}
For all $k \in \N$ it holds that $3^ka(r) \equiv_6 a(r)$.
Since $6x^{15}$ and $6x^{33}$ lie in $\langle F \rangle$, we get 
\begin{equation} \label{eq:impeq2}
 3^{l_1+k_2+1}x^{r + 18 }  +3^{2l_1 +k_2+1}x^{2r+9}  + a(r)x^{15} + a(r) x^{33} \in \langle F \rangle . 
\end{equation}
Now we have again to distinguish between some cases for $r$. 
Therefore, we consider the following table.
We read this table like the table before.  

\begin{table}[h]
\centering
\footnotesize
\setstretch{1.5}
\setlength{\tabcolsep}{2pt}
\label{secondtable}
\begin{tabular}{@{}|l||c|l|lll|lll|l|c|c|@{}}
\hline  

\setstretch{0.7} \begin{tabular}[l]{@{}l@{}}  Case: \\ $r \in $ \end{tabular}    &  $a(r)$        &  \setstretch{0.7}  \begin{tabular}[c]{@{}l@{}} Then \\ $r+18 \in $ \end{tabular}    &   \multicolumn{3}{l|}{  \setstretch{0.7} \begin{tabular}[l]{@{}l@{}} \\ Then the induction \\ hypothesis yields a \\ $l_4 \in \N_0$ such that \\$ \ldots \in \langle F \rangle$ \\ \\ \end{tabular} }   &     \multicolumn{3}{l|}{ \setstretch{0.7}  \begin{tabular}[l]{@{}l@{}} \\ Together with \eqref{eq:impeq2}, \\ Lemma \ref{lem:help3lxi} yields a \\ $t \in \N_0$ such that \\$ \ldots \in \langle F \rangle$ \\ \\ \end{tabular}}    &  $i \in$ &  $a(i)$ &  $b(i)$ \\ 
\hline
$\eqclass{72}{9}$                                              & 0             & $\eqclass{72}{27}$                                              & $3^{l_4}x^{r+18}$  &                 &                                                                                                                                                & $3^{t}x^{2r +9}$     &                  &                                                                                                                                                         & $\eqclass{72}{27}$   & 0 & 0\\
$\eqclass{72}{15}$                                             & 3             & $\eqclass{72}{33}$                                              & $3^{l_4}x^{r+18}$  &                 & $+ 3x^{33}$                                                                                                                                    & $3^{t}x^{2r +9}$     &    $+ 3x^{15}$   &                                                                                                                                                         & $\eqclass{72}{39}$   & 3 & 0\\
$\eqclass{72}{21}$                                             & 3             & $\eqclass{72}{39}$                                              & $3^{l_4}x^{r+18}$  & $+ 3x^{15}$     &                                                                                                                                                & $3^{t}x^{2r +9}$     &                  &  $+ 3x^{33}$                                                                                  							     & $\eqclass{72}{51}$   & 0 & 3\\
$\eqclass{72}{27}$                                             & 0             & $\eqclass{72}{45}$                                              & $3^{l_4}x^{r+18}$  & $+ 3x^{15}$     & $+ 3x^{33}$                                                                                                                                    & $3^{t}x^{2r +9}$     &    $+ 3x^{15}$   &  $+ 3x^{33}$                                                                                                                                            & $\eqclass{72}{63}$   & 3 & 3\\
$\eqclass{72}{33}$                                             & 0             & $\eqclass{72}{51}$                                              & $3^{l_4}x^{r+18}$  &                 & $+ 3x^{33}$                                                                                                                                    & $3^{t}x^{2r +9}$     &                  &  $+ 3x^{33}$                                                                                                                                            & $\eqclass{72}{3}$    & 0 & 3\\
$\eqclass{72}{39}$                                             & 3             & $\eqclass{72}{57}$                                              & $3^{l_4}x^{r+18}$  &                 & $+ 3x^{33}$                                                                                                                                    & $3^{t}x^{2r +9}$     &    $+ 3x^{15}$   &                                                                                                                                                         & $\eqclass{72}{15}$   & 3 & 0\\
$\eqclass{72}{45}$                                             & 3             & $\eqclass{72}{63}$                                              & $3^{l_4}x^{r+18}$  & $+ 3x^{15}$     & $+ 3x^{33}$                                                                                                                                    & $3^{t}x^{2r +9}$     &                  &                                                                                                                                                         & $\eqclass{72}{27}$   & 0 & 0\\
$\eqclass{72}{51}$                                             & 0             & $\eqclass{72}{69}$                                              & $3^{l_4}x^{r+18}$  & $+ 3x^{15}$     &                                                                                                                                                & $3^{t}x^{2r +9}$     &    $+ 3x^{15}$   &                                                                                                                                                         & $\eqclass{72}{39}$   & 3 & 0\\
$\eqclass{72}{57}$                                             & 0             & $\eqclass{72}{3}$                                               & $3^{l_4}x^{r+18}$  &                 & $+ 3x^{33}$                                                                                                                                    & $3^{t}x^{2r +9}$     &                  &   $+ 3x^{33}$                                                                                                                                           & $\eqclass{72}{51}$   & 0 & 3\\
$\eqclass{72}{63}$                                             & 3             & $\eqclass{72}{9}$                                               & $3^{l_4}x^{r+18}$  &                 &                                                                                                                                                & $3^{t}x^{2r +9}$     &    $+ 3x^{15}$   &   $+ 3x^{33}$                                                                                                                                           & $\eqclass{72}{63}$   & 3 & 3\\
$\eqclass{72}{69}$                                             & 3             & $\eqclass{72}{15}$                                              & $3^{l_4}x^{r+18}$  & $+ 3x^{15}$     &                                                                                                                                                & $3^{t}x^{2r +9}$     &                  &   $+ 3x^{33}$                                                                                                                                           & $\eqclass{72}{3}$    & 0 & 3\\
$\eqclass{72}{3}$                                              & 0             & $\eqclass{72}{21}$                                              & $3^{l_4}x^{r+18}$  & $+ 3x^{15}$     &                                                                                                                                                & $3^{t}x^{2r +9}$     &    $+ 3x^{15}$   &                                                                                                                                                         & $\eqclass{72}{15}$   & 3 & 0\\
\hline
\end{tabular}

\caption{\textit{Case:} $i \in \eqclass{12}{3}$. Distinction for $r \mod 72$ in $\eqclass{6}{3}$}
\end{table}
Hence we have that in each case there exists a $t \in \N_0$ such that 
$$3^{t} x^i + a(i)x^{15} + b(i)x^{33} \in \langle F \rangle .$$
This concludes the case $i \in \eqclass{12}{3}$.
\end{proof}

\begin{lem} \label{lem:cxiax15bx33inF}
Let $F := \{x^3 \}$. 
For all $ i \in \eqclass{6}{3}$ we have that $$3^{ \frac{s_3(i)-1}{2}} x^i +a(i)x^{15} +b(i)x^{33} $$
lies  in $\langle F \rangle $.
\end{lem}

\begin{proof}
Let $e(i ):=   \frac{s_3(i)-1}{2} $
and $c(i) := 3^{e(i)} $.
Since $i \in \eqclass{6}{3} $ and $i$ is therefore odd, the digit sum $s_3(i)$ is always odd.
We proceed by induction on the digit sum of $i$.
If $s_3(i) = 1$, $i$ is a power of 3.
Hence by Lemma \ref{lem:x3elemofF}, $x^i \in \langle F \rangle$.
If $i=3$, $a(i) = b(i) =0$.
If $i=3^n$ with $n \geq 2$, then $i \in \eqclass{72}{9, 27}$ and therefore $a(i) = b(i) = 0.$
Now we assume $s_3(i) > 1$.
Then we choose $h_1, h_2, k \in \N$ such that $ i= 3^{h_1}+ 3^{h_2} + k$ with $ k \in \eqclass{6}{3}$ and $s_3(i) = s_3(k) + 2$.
We have by the induction hypothesis, $p(x) := c(k)x^k + a(k)x^{15}+b(k)x^{33} $
lies in $\langle F \rangle$.
For all $q(x) \in \langle F \rangle$, we have
\begin{equation*}
\begin{aligned}
 x^3 \circ (x^{3^{h_1}} + x^{3^{h_2}} + q(x) ) = 
 q(x)^3 &+ x^{3^{1 + h_1}} +  x^{3^{1 + h_2}} \\
 & + 3 q(x) x^{2 \cdot 3^{h_1}} + 3 q(x)^2 x^{3^{h_1}}\\
 &+ 3 q(x) x^{2 \cdot 3^{h_2}} + 3 q(x)^2 x^{3^{h_2}} \\
 &+ 3 x^{2 \cdot 3^{h_1} + 3^{h_2}} + 3 x^{3^{h_1} + 2 \cdot 3^{h_2}} \\
 &+  6 q(x) x^{3^{h_1} + 3^{h_2}} \in \langle F \rangle.
 \end{aligned}
\end{equation*}
We have $q(x) \in \langle F \rangle$, hence also $q(x)^3 \in \langle F \rangle.$
By Lemma \ref{lem:x3elemofF} we also have that $x^{3^{h_1+1}}$, $x^{3^{h_2+1}}$,  $3x^{2 \cdot 3^{h_1}+3^{h_2}} +3x^{3^{h_1}+2 \cdot 3^{h_2}} $,
$3 q(x) x^{2 \cdot 3^{h_1}} + 3 q(x)^2 x^{3^{h_1}}$ and $3 q(x) x^{2 \cdot 3^{h_2}} + 3 q(x)^2 x^{3^{h_2}}$ lie in $ \langle F \rangle$.
Since $x^3 \circ (x^{3^{h_1}} + x^{3^{h_2}} + q(x))$ lies in $\langle F \rangle$, we obtain $6q(x)x^{3^{h_1}+3^{h_2}} \in \langle F \rangle$.
Therefore, we get by setting $q(x):= p(x)$ that 
\begin{equation}\label{eq:gle}
 6 c(k) x^{k+3^{h_1}+3^{h_2}} + 6 a(k) x^{15+3^{h_1}+3^{h_2}} + 6 b(k) x^{33+3^{h_1}+3^{h_2}} \in \langle F \rangle. 
\end{equation}
Now we want to prove that $6 a(k) x^{15+3^{h_1}+3^{h_2}}$ and $6 b(k) x^{33+3^{h_1}+3^{h_2}}$ lie in $\langle F \rangle$.
By Lemma \ref{lem:x3elemofF} we know that $3x^{15} +3x^{21} \in \langle F \rangle$.
Setting $q(x)  := 3x^{15} +3x^{21} $ we get $18x^{21+3^{h_1}+3^{h_2}} + 18x^{15+3^{h_1}+3^{h_2}} \in \langle F \rangle$.
Since $ s_3(21 +3^{h_1} +3^{h_2}) \leq 5 $ and $s_3(21 +3^{h_1} +3^{h_2})$ is odd we get $c(21 +3^{h_1} +3^{h_2}) \in \{1, 3, 9\} $.
In all cases we get that $18x^{21+3^{h_1}+3^{h_2}} \in \langle F \rangle $ and hence $18x^{15+3^{h_1}+3^{h_2}} \in \langle F \rangle$.
By Lemma \ref{lem:x3elemofF} we know $3x^{33} +3x^{57} \in \langle F \rangle$.
Setting $ q(x) := 3x^{33} +3x^{57} $ yields $18x^{57+3^{h_1}+3^{h_2}} + 18x^{33+3^{h_1}+3^{h_2}} \in \langle F \rangle$.
Since $ s_3(57 +3^{h_1} +3^{h_2}) \leq 5 $ and $s_3(57 +3^{h_1} +3^{h_2})$ is odd we get $c(57 +3^{h_1} +3^{h_2}) \in \{1, 3, 9\}$.
In all cases we get that $18x^{57+3^{h_1}+3^{h_2}} \in \langle F \rangle $ and hence $18x^{33+3^{h_1}+3^{h_2}} \in \langle F \rangle$.
Since $a(k) \in \{0, 3\}$, $b(k) \in \{0, 3\}$, $18x^{15+3^{h_1}+3^{h_2}} \in \langle F \rangle$ and $18x^{33+3^{h_1}+3^{h_2}} \in \langle F \rangle$ 
we get from \eqref{eq:gle} that $6 c(k) x^{k+3^{h_1}+3^{h_2}} \in \langle F \rangle$.
This implies
$$ 2 \cdot 3^{1+e(k)} x^{3^{h_1}+3^{h_2}+k} \in \langle F \rangle. $$
Since $s_3(3^{h_1} +3^{h_2}+k) = 2 + s_3(k) $, $e(3^{h_1} +3^{h_2}+k) = e(k) + 1$.
Therefore,
\begin{equation} \label{eq:cxiax15bx33inFeq1}
  2 \cdot c(3^{h_1} +3^{h_2} +k )x^{3^{h_1}+3^{h_2}+k}\in \langle F \rangle.
\end{equation}
Hence, $2c(i)x^i \in \langle F \rangle$. 
By Lemma \ref{lem:3lxiax15bx33inF} we have that there exists a
$j \in \N_0$ such that 
\begin{equation} \label{eq:cxiax15bx33inFeq2}
 3^{j} x^i +a(i)x^{15} +b(i)x^{33} \langle F \rangle.
\end{equation}
We have that $c(i)= 3^{e(i)}$ and therefore
$c(i)$ is a multiple of $\gcd(2 \cdot c(i), 3^j)$.
Then there are $ m, n \in \Z$ such that $ m \cdot 2 \cdot c(i) + n \cdot 3^j = c(i)$ and therefore we get with 
\eqref{eq:cxiax15bx33inFeq1} and \eqref{eq:cxiax15bx33inFeq2} that
\begin{equation}
c(i) x^i + n \cdot a(i)x^{15} +  n\cdot b(i)x^{33}\in \langle F \rangle.
\end{equation}
Since $c(i)$ is odd and $ m \cdot 2 \cdot c(i)$ is even we have that $n$ is odd.
Therefore there exists $n_{1} \in \Z$ such that $n = 2n_{1} +1$.
Since $6x^{15}$ and $6x^{33}$ lie in $\langle F \rangle$ and $2\cdot n_{1}\cdot a(i)$ and $2\cdot n_{1} \cdot b(i)$ are multiples of 6, 
we also get $2\cdot n_{1}\cdot a(i)x^{15} \in \langle F \rangle$ and  $2\cdot n_{1} \cdot b(i)x^{33} \in \langle F \rangle$.
Now we compute 
$$c(i) x^i + (2n_{1}+1) \cdot a(i)x^{15} +  (2n_{1}+1) \cdot b(i)x^{33} - 2\cdot n_{1}\cdot a(i)x^{15} - 2 \cdot n_{1} \cdot b(i)x^{33} \in \langle F \rangle $$
and therefore we get
$$c(i)x^i +a(i)x^{15} +b(i)x^{33} \in \langle F \rangle.$$
\end{proof}

\begin{lem} \label{lem:calc24}
If $i\in \eqclass{6}{3} $ then the following hold: 
\begin{enumerate}
 \item $3i \equiv_{24} 21 \Leftrightarrow i \equiv_{24} 15  $.
 \item $3i \equiv_{24} 15 \Leftrightarrow i \equiv_{24} 21  $.
\end{enumerate}
\end{lem}

\begin{proof}
 For all $i \in \Z$ it holds that 
  $$3i \equiv_{24} 21 \Leftrightarrow (i \equiv_{24} 7 \vee i \equiv_{24}  15 \vee  i \equiv_{24} 23 )$$ and
  $$3i \equiv_{24} 15 \Leftrightarrow (i \equiv_{24} 5 \vee i \equiv_{24}  13 \vee  i \equiv_{24} 21 ).$$
Since $i\in \eqclass{6}{3} $ we are done.
\end{proof}

\begin{lem}\label{lem:equiv24}
Let $ p = \sum_{i=0 }^n c_ix^i \in M$ and let  $P =\sum_{j=0 }^m C_jx^j$
with $C_0, \ldots , C_m \in \Z $ such that $P = p^3$. 
Then 2 divides $\sum_{j \in A} C_j$.
\end{lem}

\begin{proof}
We have $6c_ic_jc_k \equiv_2 0$ for all $i, j, k \in \N_0$ and $3c_i^2c_j + 3c_ic_j^2 \equiv_2 3c_ic_j + 3c_ic_j \equiv_2 2 \cdot (3c_ic_j) \equiv_2 0 $ for all $i, j \in \N_0 $.
Furthermore, 
$$ \SUM{j \in A}{}{C_j} = \SUM{\noSp{r \in \eqclass{6}{3}}{3r \in A}}{}{c_r^3} +
\SUM{\noSp{\noSp{r, s \in \eqclass{6}{3}}{r \neq s}}{2r+s \in A}}{}{3c_r^2c_s} +
\SUM{\noSp{\noSp{r, s, t \in \eqclass{6}{3}}{r < s < t}}{r+s +t \in A}}{}{6c_rc_sc_t}.
$$
We compute
$$ \SUM{\noSp{r \in \eqclass{6}{3}}{3r \in A}}{}{c_r^3} +
\SUM{\noSp{\noSp{r, s \in \eqclass{6}{3}}{r \neq s}}{2r+s \in A}}{}{3c_r^2c_s} +
\SUM{\noSp{\noSp{r, s, t \in \eqclass{6}{3}}{r < s < t}}{r+s +t \in A}}{}{6c_rc_sc_t} \equiv_2 
 \SUM{\noSp{r \in \eqclass{6}{3}}{3r \in A}}{}{c_r^3},$$ and since $ c^3 \equiv_2 c$ we get
$$ \SUM{\noSp{r \in \eqclass{6}{3}}{3r \in A}}{}{c_r^3} 
\equiv_2   \SUM{\noSp{r \in \eqclass{6}{3}}{3r \in A}}{}{c_r}. $$
By \ref{lem:calc24} we get that 
$$ \SUM{\noSp{r \in \eqclass{6}{3}}{3r \in A}}{}{c_r} = 
\SUM{r \in A}{}{c_r}  $$
and this is divisible by 2 since $ p \in M$.
\end{proof}

\begin{lem}\label{lem:calc72}
Let $k\in \eqclass{6}{3}$. Then it holds that: 
$3k \in \eqclass{72}{3, 33 , 45, 51, 57, 63}$ if and only if $  k \in \eqclass{72}{15, 21, 39, 45, 63, 69 } $.
\end{lem}

\begin{proof}
  The following hold for all $k \in \Z$:
  \begin{eqnarray*}
   3k \equiv_{72} 3 \Leftrightarrow  (k \equiv_{72} 1 \vee k \equiv_{72} 25 \vee k \equiv_{72} 49).\\
   3k \equiv_{72} 33 \Leftrightarrow (k \equiv_{72} 11\vee k \equiv_{72} 35 \vee k \equiv_{72} 59).\\
   3k \equiv_{72} 45 \Leftrightarrow (k \equiv_{72} 15 \vee k \equiv_{72} 39 \vee k \equiv_{72} 63).\\
   3k \equiv_{72} 51 \Leftrightarrow (k \equiv_{72} 17 \vee k \equiv_{72} 41 \vee k \equiv_{72} 65).\\
   3k \equiv_{72} 57 \Leftrightarrow (k \equiv_{72} 19 \vee k \equiv_{72} 43 \vee k \equiv_{72} 67). \\
   3k \equiv_{72} 63 \Leftrightarrow (k \equiv_{72} 21 \vee k \equiv_{72} 45 \vee k \equiv_{72} 69).
 \end{eqnarray*}
Since $k \in \eqclass{6}{3} $ the result follows.
\end{proof}

\begin{lem}\label{lem:equiv72}
Let $ p = \sum_{i=0 }^n c_ix^i \in M$ and let  $P =\sum_{j=0 }^m C_jx^j$
with $C_0, \ldots , C_m \in \Z $ such that $P = p^3$ . 
Then 2 divides $ \sum_{j \in B}C_j$.
\end{lem}

\begin{proof}
We have again $6c_ic_jc_k\equiv_2 0$ for all $i, j, k \in \N_0$ and $3c_i^2c_j + 3c_ic_j^2 \equiv_2 0 $ for all $i, j \in \N_0 $.
Furthermore, it holds that
\begin{eqnarray*}
\SUM{j \in B}{}{C_j} =
\SUM{\noSp{r \in \eqclass{6}{3}}{3r \in B}}{}{ c_r^3 } +
\SUM{\noSp{r,s \in \eqclass{6}{3}}{\noSp{r \neq s}{ 2r +s \in B}}}{}{ 3c_r^2c_s } +
\SUM{\noSp{r,s, t \in \eqclass{6}{3}}{\noSp{r < s < t}{ 2r +s \in B}}}{}{ 6c_rc_sc_t }.
\end{eqnarray*}
We compute
\begin{eqnarray*}
\SUM{\noSp{r \in \eqclass{6}{3}}{3r \in B}}{}{ c_r^3 } +
\SUM{\noSp{r,s \in \eqclass{6}{3}}{\noSp{r \neq s}{ 2r +s \in B}}}{}{ 3c_r^2c_s } +
\SUM{\noSp{r, s, t \in \eqclass{6}{3}}{\noSp{r < s < t}{ 2r +s \in B}}}{}{ 6c_rc_sc_t } \equiv_2  
\SUM{\noSp{r \in \eqclass{6}{3}}{3r \in B}}{}{c_r^3},
\end{eqnarray*} 
and since $ c^3 \equiv_2 c$ we get
$$ \SUM{\noSp{r \in \eqclass{6}{3}}{3r \in B}}{}{c_r^3} 
 \equiv_2   \SUM{\noSp{r \in \eqclass{6}{3}}{3r \in B}}{}{c_r} $$
By \ref{lem:calc72} we get that
$$ \SUM{\noSp{r \in \eqclass{6}{3}}{3r \in B}}{}{c_r}
\equiv_2  
\SUM{j \in \eqclass{72}{15, 21, 39, 45, 63, 69}}{}{c_j}.
$$
We can simplify $ \SUM{j \in \eqclass{72}{15, 21, 39, 45, 63, 69}}{}{c_j}=  \SUM{j \in A}{}{c_j}.$ 
Since $p \in M$ we have 2 divides $\SUM{j \in A}{}{c_j}$ and therefore we are done.
\end{proof}

\begin{lem} \label{lem:x3MsubF} 
 Let $F := \{ x^3 \} $.
 Then $M \subseteq \langle  F \rangle$.
\end{lem}

\begin{proof}
 We have to show for all $n \in \N $ and for all $\sum_{i=0}^{n}c_ix^i \in M $, we have $\sum_{i=0}^{n}c_ix^i \in \langle F \rangle$.
 We proceed by induction on $n$.
 For the induction basis let $n=3$. 
 Since $x^3 \in \langle F \rangle $ we also have $c_3x^3 \in \langle F \rangle$.
 Now let $n>3$ and let  $\sum_{i=0}^{n} c_ix^i \in M$.
 If $n = 9$, Lemma \ref{lem:x3elemofF} yields that $x^9 \in \langle F \rangle $ and hence $c_9x^9 + c_3x^3 \in \langle F \rangle$.
 If $n=15$ we have to show that 
 \begin{equation} \label{eq:elem1ofF}
  c_{15}x^{15} + c_9x^9 + c_3x^3 \in \langle F \rangle,
 \end{equation}
 knowing by the assumptions on $c_{15}$ that $3^{ \frac{s_3(15)-1}{2}} = 3$ divides $ c_{15}$ and $2$ divides $c_{15}$.
 Hence, $6$ divides $c_{15}$.
 By Lemma \ref{lem:x3elemofF} we have $6x^{15} \in \langle F \rangle$ and since $c_3x^3$ and $c_9x^9$ lie in $\langle F \rangle$ we get \eqref{eq:elem1ofF}.
 If $n=21$ we have to show that 
 \begin{equation} \label{eq:elem2ofF}
  c_{21}x^{21} + c_{15}x^{15} + c_9x^9 + c_3x^3 \in \langle F \rangle,
 \end{equation}
 knowing by the assumptions on $c_{21}$ and $c_{15}$ that $3^{ \frac{s_3(21)-1}{2}} = 3$ divides $c_{21}$, $3$ divides $ c_{15}$ and $2$ divides $c_{21} + c_{15}$.\\
 \textit{Case: $6$ divides $c_{21}$ and $6$ divides $c_{15}$}:
 Setting $p(x):=x^9$ and $a:=1$, Lemma \ref{lem:x3elemofF} yields that $6x^{21} \in \langle F \rangle$. 
 Since $6x^{15}$, $c_3x^3$ and $c_9x^9$ lie in $\langle F \rangle$ we get \eqref{eq:elem2ofF}. \\
 \textit{Case: $6$ divides $c_{21} +c_{15}$ and $6$ does not divide $c_{21}$}:
 By Lemma \ref{lem:x3elemofF} we have ${3x^{21} +3 x^{15} \in \langle F \rangle}$ and since
 $6x^{21}$ and $6x^{15}$ lie in $\langle F \rangle$ and both $c_{21}$ and $c_{15}$ are elements of $\eqclass{6}{3}$, we obtain
 $c_{21}x^{21} + c_{15}x^{15} \in \langle F \rangle$.
 Since $c_3x^3$ and $c_9x^9$ lie in $\langle F \rangle$ we get \eqref{eq:elem2ofF}.\\
 If $n=27$ we have to show that 
 \begin{equation} \label{eq:elem3ofF}
 c_{27}x^{27} + c_{21}x^{21} + c_{15}x^{15} + c_9x^9 + c_3x^3 \in \langle F \rangle,
 \end{equation}
 knowing by the assumptions on $c_{21}$ and $c_{15}$ that $3$ divides $c_{21}$, $3$ divides $ c_{15}$ and $2$ divides $c_{21} + c_{15}$.
 By Lemma \ref{lem:x3elemofF} we know that $x^{27} \in \langle F \rangle$ and therefore also $c_{27}x^{27} \in \langle F \rangle$.
 Subtracting $c_{27}x^{27}$ from \ref{eq:elem3ofF} we get $ c_{21}x^{21} + c_{15}x^{15} + c_9x^9 + c_3x^3  \in M $.
 By the induction hypothesis, $ c_{21}x^{21} + c_{15}x^{15} + c_9x^9 + c_3x^3  \in \langle F \rangle$.
 Hence, we get \eqref{eq:elem3ofF}.
 If $n=33$ we have to show that 
 \begin{equation} \label{eq:elem4ofF}
 c_{33}x^{33} + c_{27}x^{27} + c_{21}x^{21} + c_{15}x^{15} + c_9x^9 + c_3x^3 \in \langle F \rangle,
 \end{equation}
 knowing by the assumptions on $c_{33}$, $c_{21}$ and $c_{15}$ that $3^{ \frac{s_3(33)-1}{2}} =3 $ divides $c^{33}$, $3$ divides $c_{21}$, 
 $3$ divides $ c_{15}$, $2$ divides $c_{21} + c_{15}$ and $2$ divides $c_{33}$.
  By Lemma \ref{lem:x3elemofF} we know that $6x^{33} \in \langle F \rangle$ and therefore also $c_{33}x^{33} \in \langle F \rangle$.
  The polynomial $c_{33}x^{33}$ also lies in $M$ by the definition of $M$. 
 By the induction hypothesis we get that $c_{27}x^{27} + c_{21}x^{21} + c_{15}x^{15} + c_9x^9 + c_3x^3 \in \langle F \rangle$.
 Hence we get \eqref{eq:elem4ofF}. \\
 If $n > 33$ we define $p(x):=\sum_{i=0}^{n}c_ix^i$.
 If $n \not\in \eqclass{6}{3}$ we have $c_n = 0$ and therefore the induction hypothesis yields that $p(x) \in \langle F \rangle$.
 Now let $n \in \eqclass{6}{3} $.
 Then we know by Lemma \ref{lem:cxiax15bx33inF} that $q(x):=3^{ \frac{s_3(n)-1}{2}}x^{n} + a(n)x^{15} + b(n)x^{33}$ lies in $\langle F \rangle$.
 Since $c_n$ is a multiple of $3^{ \frac{s_3(n)-1}{2}}$ there exists $m \in \Z$ such that  $c_n = m \cdot 3^{ \frac{s_3(n)-1}{2}}$.
 Since $M$ is closed under $+, -$, and since $q(x)$ also lies in $M$ we get
  $p(x) - m \cdot q(x) \in M$.
 Since $\deg(p(x) - m \cdot q(x) ) < n$ we know by the induction hypothesis that $p(x) - m \cdot q(x) \in \langle F \rangle $.
 Since $ m \cdot q(x) \in \langle F \rangle $ we also get $p(x) \in \langle F \rangle$.
\end{proof}

\begin{thm} \label{thm:x3}
A polynomial $p = \sum_{ i=0 }^n c_ix^i \in \Z[x]$ lies in the subnearring of \\ ${\algop{\Z[x]}{+, \circ}}$ 
that is generated by $\{ x^3 \}$ if and only if for all $i \in \N_0$ the following conditions hold:
\begin{enumerate}
\item If $ i \not\in \eqclass{6}{3}$ then $c_i =0$.
\item If $ i \in \eqclass{6}{3}$ then $3^{ \frac{s_3(i)-1}{2}}$ divides $ c_i.$ 
\item  $2$ divides $ \SUM{j \in \eqclass{24}{15, 21}}{}{c_j} $.
\item  $ 2 $ divides $\SUM{\noSp{ j \in \eqclass{72}{3, 33 , 45, 51, 57, 63} }{j \neq 3} }{}{c_j} $. 
\end{enumerate}
\end{thm}

\begin{proof}
We use Lemma \ref{lem:genLem1} with $F := \{ x^3 \}$.
We know that $F \subseteq M$ and $M$ is closed under $+, -$.
Let $P= \sum_{ i=0 }^m C_ix^i \in \Z[x]$ such that $P=p^3$.
It can be easily verified for all $i \leq m$ that if $ i \not\in \eqclass{6}{3}$ then $C_i =0$
and if $ i \in \eqclass{6}{3}$ then $3^{ \frac{s_3(i)-1}{2}}$ divides $C_i$.
Therefore, with Lemma \ref{lem:equiv24} and Lemma \ref{lem:equiv72} it follows that $F \circ M \subseteq M$. 
By Lemma \ref{lem:x3MsubF} we also have  $M \subseteq \langle  F \rangle$.
Therefore, Lemma \ref{lem:genLem1} yields $M= \langle F \rangle$.
\end{proof}

\begin{thm} \label{thm:xx3}
A polynomial $p = \sum_{ i=0 }^n c_ix^i \in \Z[x]$ lies in the subnearring of \\ ${\algop{\Z[x]}{+, \circ}}$ 
that is generated by $\{x, x^3 \}$ if and only if for all $i \in \N_0$ the following conditions hold:
\begin{enumerate}
 \item If $ i \not\in \eqclass{2}{1}$  then $c_i = 0$.
 \item If $ i  \in \eqclass{2}{1}$  then $ 3^{\frac{s_3(i) -1}{2}}$ divides $c_i$.
 \item $2$ divides $ \SUM{j \in \eqclass{8}{5, 7} }{}{c_j} $.
 \item $2$ divides $ \SUM{ \noSp{j \in \eqclass{24}{1, 11, 15, 17, 19, 21 } }{j \neq 1} }{}{c_j} $ .
\end{enumerate}
\end{thm}

\begin{proof}
``$\Rightarrow$'':  
By Lemma \ref{lem:genLem3} we know that $q := p(x^3) \in  \langle \{x^3\} \rangle$.
Let $q(x) = \sum_{j =0}^{3n} \tilde{c_j} x^j$. 
For condition (1) let $i$ be an even number.
Then by Theorem \ref{thm:x3}, $c_i = \tilde{c}_{3i} = 0$ because $3i \not \in \eqclass{6}{3}$.
For condition (2) let $i$ be odd.
Then $3i \in \eqclass{6}{3}$, hence $3^{\frac{s_3(3i) -1}{2}} \mid \tilde{c}_{3i} = c_i$.
Now the result follows from $s_3(3i)=s_3(i)$.
Now we show (3).
By Theorem \ref{thm:x3}, $2 \mid \sum_{j \in \eqclass{24}{15, 21}} \tilde{c}_j$ and
since $c_j = \tilde{c}_{3j}$ for all $j \in \{0, \ldots, n\}$ we also have
that $2$ divides $\sum_{j \in \eqclass{24}{15, 21}} c_{\frac{j}{3}} =  \sum_{j \in \eqclass{8}{5, 7}} c_j$.
Condition (4) is proved similarly. 
``$\Leftarrow$'': 
By Theorem \ref{thm:x3}, $p(x^3) = \sum_{i = 0}^{n} c_i x^{3i}$ lies in $\langle \{x^3 \} \rangle$ 
because $s_3(3i) = s_3(i)$ and 
$c_i = \tilde{c}_{3i}$ for all $i \in \N_0$.
Therefore Lemma \ref{lem:genLem3} yields $p \in \langle \{x, x^3\} \rangle$.
\end{proof}

\begin{thm} \label{thm:1x3}
A polynomial $p = \sum_{ i=0 }^n c_ix^i \in \Z[x]$ lies in the subnearring of $\algop{\Z[x]}{+, \circ} $ that is generated
by $\{1, x^3 \}$ if and only if for all $i \in \N_0$ the following conditions hold:
\begin{enumerate}
 \item If $ i \not\in \eqclass{3}{0}$  then $c_i = 0$.
 \item If $ i  \in \eqclass{3}{0}$  then $ 3^{\lfloor \frac{s_3(i)}{2} \rfloor}$ divides $c_i$.
\end{enumerate}
\end{thm}

\begin{proof}
``$\Rightarrow$'':  
By Lemma \ref{lem:genLem2} we know that there exists $q \in \langle \{1, x, x^3 \} \rangle$ 
such that $p = q(x^3)$.
By Theorem 1.3 of~\cite{Ai:GPUF} it follows that 
$p$ satisfies the conditions (1) and (2) 
because
$s_3(3i) = s_3(i)$ for all $i \in \N_0$.
``$\Leftarrow$'': 
Let $q \in \Z[x]$ such that $q \circ x^3 = p$.
Then $q$ lies in $ \langle \{1, x, x^3 \} \rangle$ by Theorem 1.3 of~\cite{Ai:GPUF} because $s_3(3i) = s_3(i)$ for all $i \in \N_0$.
Therefore Lemma \ref{lem:genLem2} yields $p \in \langle \{1, x^3\} \rangle$.
\end{proof}

\section*{Acknowledgements}
The authors thank Gleb Pogudin for discussions leading to Theorem~\ref{thm:indep}.

\def\cprime{$'$}
\providecommand{\bysame}{\leavevmode\hbox to3em{\hrulefill}\thinspace}
\providecommand{\MR}{\relax\ifhmode\unskip\space\fi MR }
\providecommand{\MRhref}[2]{%
  \href{http://www.ams.org/mathscinet-getitem?mr=#1}{#2}
}
\providecommand{\href}[2]{#2}

\end{document}